\documentclass[12pt,reqno]{amsart}

\usepackage{amsmath,amsthm,amssymb,comment,fullpage}
\usepackage{outlines}
\usepackage{enumerate}
\usepackage{pgfplots}
\usepackage{times}
\usepackage[T1]{fontenc}
\usepackage{mathrsfs}
\usepackage{latexsym}
\usepackage{caption}
\usepackage{amsmath,amsfonts,amsthm,amssymb,amscd}
\input amssym.def
\input amssym.tex
\usepackage{color}
\usepackage{hyperref}
\usepackage{url}
\usepackage{breakurl}
\usepackage{algorithmicx}
\usepackage{comment}
\newcommand{\bburl}[1]{\textcolor{blue}{\url{#1}}}
\usepackage{mathtools}  
\usepackage[linesnumbered]{algorithm2e} 

\usepackage{pdflscape} 
\usepackage{fancyhdr} 
\usepackage[margin=1in]{geometry} 

\usepackage[
backend=biber,
maxbibnames=99,
giveninits=true,
style=alphabetic,
isbn=false,
doi=false,
url=false,
]{biblatex}

\usepackage{tikz} 
\usetikzlibrary{positioning} 

\makeatletter
\newcommand{\addresseshere}{%
  \enddoc@text\let\enddoc@text\relax
}
\makeatother

\numberwithin{equation}{section}

\newtheorem{thm}{Theorem}[section]

\theoremstyle{plain}

\newtheorem{corollary}[thm]{Corollary}

\newtheorem{definition}[thm]{Definition}

\newtheorem{lemma}[thm]{Lemma}
\newtheorem{proposition}[thm]{Proposition}
\newtheorem{theorem}[thm]{Theorem}

\newtheorem{observation}[thm]{Observation}

\theoremstyle{definition}
\newtheorem{example}[thm]{Example}
\newtheorem{question}[thm]{Question}

\newcommand\be{\begin{equation}}
\newcommand\ee{\end{equation}}
\newcommand\bea{\begin{eqnarray}}
\newcommand\eea{\end{eqnarray}}
\newcommand\bi{\begin{itemize}}
\newcommand\ei{\end{itemize}}
\newcommand\ben{\begin{enumerate}}
\newcommand\een{\end{enumerate}}
\newcommand\bc{\begin{center}}
\newcommand\ec{\end{center}}
\newcommand\ba{\begin{array}}
\newcommand\ea{\end{array}}

\newcommand{\hr}[1]{\href{#1}{\url{#1}}}

\title{A geometric perspective on the MSTD question}

\author{Steven J. Miller}
\email{\textcolor{blue}{\href{mailto:sjm1@williams.edu, Steven.Miller.MC.96@aya.yale.edu}{sjm1@williams.edu,Steven.Miller.MC.96@aya.yale.edu}}}
\address{Department of Mathematics and Statistics, Williams College, Williamstown, MA 01267}

\author{Carsten Peterson}
\email{\textcolor{blue}{\href{mailto:carstenp@umich.edu}{carstenp@umich.edu}}}
\address{Department of Mathematics, University of Michigan, Ann Arbor, MI 48109}

\thanks{The first named author was partially supported by NSF Grants DMS1265673 and DMS1561945, and the second by DMS1347804. We thank Mel Nathanson, Carsten Sprunger and Roger Van Peski for helpful conversations.}



\date{\today}

\begin{document}

\maketitle


\begin{abstract}
A more sums than differences (MSTD) set $A$ is a subset of $\mathbb{Z}$ for which $|A+A| > |A-A|$. Martin and O'Bryant used probabilistic techniques to prove that a non-vanishing proportion of subsets of $\{1, \dots, n\}$ are MSTD as $n \to \infty$. However, to date only a handful of explicit constructions of MSTD sets are known. We study finite collections of disjoint intervals on the real line, $\mathbb{I}$, and explore the MSTD question for such sets, as well as the relation between such sets and MSTD subsets of $\mathbb{Z}$. In particular we show that every finite subset of $\mathbb{Z}$ can be transformed into an element of $\mathbb{I}$ with the same additive behavior. Using tools from discrete geometry, we show that there are no MSTD sets in $\mathbb{I}$ consisting of three or fewer intervals, but there are MSTD sets for four or more intervals. Furthermore, we show how to obtain an infinite parametrized family of MSTD subsets of $\mathbb{Z}$ from a single such set $A$; these sets are parametrized by lattice points satisfying simple congruence relations contained in a polyhedral cone associated to $A$.
\end{abstract}

\section{Introduction} \label{sec:intro}

Let $A$ be a finite subset of the integers. The \textbf{sumset} and \textbf{difference set} are defined, respectively, as:
\begin{gather}
A+A\ :=\ \{a_1 + a_2: a_1, a_2 \in A\}, \\
A-A\ :=\ \{a_1 - a_2: a_1, a_2 \in A\}.
\end{gather}
If $|A-A| > |A+A|$, we say that the set $A$ is \textbf{difference dominant}. If $|A+A| > |A-A|$, we say that $A$ is \textbf{sum dominant} or, following the terminology of \cite{nathanson1}, a \textbf{more sums than differences (MSTD) set}. If $|A-A| = |A+A|$, we say $A$ is \textbf{balanced}.

Due to the commutativity of addition, there is a lot of redundancy in $A+A$, and its size is at most $\binom{n}{2} + n = \frac{n(n+1)}{2}$ where $n = |A|$ (with equality being achieved with a geometric progression, for example). In the difference set, although 0 can be represented in numerous ways (e.g. $0 = a_i - a_i$ for any $i$), as subtraction is not commutative there are at most $n^2 - n + 1$ elements in $A-A$ (again with equality being achieved when $A$ is a geometric progression). 

Since the difference set has the potential to be much larger than the sumset, we might naively believe that in general the difference set is larger. A well-known example of an MSTD set, whose exact origin is not clear, is $\{0$, $2$, $3$, $4$, $7$, $11$, $12$, $14\}$; others were given in \cite{marica} and \cite{freiman_pigarev}. Ruzsa \cite{ruzsa1, ruzsa2, ruzsa3} used probabilistic techniques to prove the existence of many MSTD sets. Though Roesler \cite{roesler} was able to show that the average size of the difference set is larger than the average value of the sumset for subsets of $[n] := \{1, \dots, n\}$, surprisingly, Martin and O'Bryant \cite{martin_ob} proved using probabilistic techniques that for all $n \geq 15$, the probability of being MSTD among subsets of $[n]$ is at least $2 \times 10^{-7}$. Zhao \cite{zhao1} showed that the probability of being sum dominant converges to a limit as $n \to \infty$ and that this limiting probability is at least $4.28 \times 10^{-4}$. Based on computer simulation, we expect the true limiting value to be around $4.5 \times 10^{-4}$. If, however, we independently choose elements of $[n]$ to be in $A$ with probability $p(n)$ which tends to zero with $n$, then Hegarty and Miller \cite{HegM09} proved that with probability 1 such a set is difference dominated.

In the last decade there has been considerable interest in explicit constructions of MSTD sets (which in light of \cite{martin_ob} must exist in large numbers). A well known technique for producing an infinite family of MSTD sets from a single one is base expansion: let $A = \{a_1, \dots, a_n\}$ be an MSTD set. For each $t \in \mathbb{N}$ and some fixed $m > 2 * \max\{|a_i|: i \in [n]\}$, define $A_t$ as
\begin{gather}
A_t\ :=\ \left\{\sum_{i = 1}^k a_j m^{i-1} : j \in [n], k \in [t]\right\}.
\end{gather}
Then $|A_t \pm A_t| = |A \pm A|^t$, thus leading to a parametrized infinite family of MSTD sets, though of extremely low density. Nathanson in \cite{nathanson2} asks if there are other parametrized families of MSTD sets. Hegarty \cite{hegarty} and Nathanson provide a positive answer; in both cases their ideas involve taking some set which is symmetric (and thus balanced) and perturbing it slightly so as to increase the number of sums while keeping the number of differences the same. Hegarty \cite{hegarty} then posits: ``More interesting, though, would be to have explicit examples of MSTD sets which are, in some meaningful sense, `radically' different from some perturbed symmetric set.''

A new method of explicitly constructing MSTD sets was found by Miller, Orosz and Schneierman \cite{mos}. Their idea is to find sets whose sumset contains all possible sums (i.e., all integers between $2 a_1$ and $2 a_n$). Then by appropriately adding elements to the fringes of such a set, one obtains an MSTD set. This technique was furthered by Zhao \cite{zhao2}, who found a larger class of sets whose sumset is as large as possible. These methods yield densities on the order of $1/n^r$ (the ideas in \cite{mos} yield $r=2$, while those in \cite{zhao2} give $r=1$).

We introduce another ``radically'' different way of constructing MSTD sets. The heuristic behind our techniques is that the property of being an MSTD set should be ``stable'' under small perturbations. In order to make this notion rigorous, we must pass from the realm of the discrete to the realm of the continuous (but we will ultimately return to the discrete setting). Let $\mathbb{I}$ denote the set of all collections of finitely many disjoint open intervals on the real line, and let $\mathbb{I}_n$ denote the set of all collections of $n$ disjoint open intervals on the real line \footnote{As if often the case when dealing with measure theoretic arguments, it does not make a meaningful difference if we use open or closed intervals (or even half-open/half-closed intervals). Certain arguments are cleaner if one uses one or the other, and thus in this paper we shall sometimes assume that the elements of $\mathbb{I}$ and $\mathbb{I}_n$ consist of collections of open intervals, and at other times closed intervals.}. For each $\mathcal{A} \in \mathbb{I}$, we define $\mathcal{A} + \mathcal{A}$ and $\mathcal{A} - \mathcal{A}$ as in the discrete case. However, we are no longer interested in the cardinality of these sets, but rather in the (Lebesgue) measure, $\mu$. We say that $\mathcal{A}$ is difference dominant, sum dominant, or balanced if $\mu(\mathcal{A} - \mathcal{A}) > \mu(\mathcal{A} + \mathcal{A})$, $\mu(\mathcal{A} - \mathcal{A}) < \mu(\mathcal{A} + \mathcal{A})$, or $\mu(\mathcal{A} - \mathcal{A}) = \mu(\mathcal{A} + \mathcal{A})$, respectively.

The foundational result of this paper is the following, which is proven in Section \ref{sec:d2c}.

\begin{theorem}\label{basic:thm}
Let $A \subset \mathbb{Z}$ with $|A| < \infty$. Then, there exists an $\mathcal{A} \in \mathbb{I}$ such that $|A + A| = \mu(\mathcal{A} + \mathcal{A})$ and $|A - A| = \mu(\mathcal{A} - \mathcal{A})$.
\end{theorem}

We will show that the construction of $\mathcal{A}$ from $A$ is very natural and straightforward.

Theorem \ref{basic:thm} justifies the study of the additive behavior of elements in $\mathbb{I}$ as a means to study the additive behavior of subsets of $\mathbb{Z}$. The space $\mathbb{I}_n$ (and related spaces) have a natural topology, and thus we have the utility of continuity arguments at our disposal in this setting. In the latter half of Section \ref{sec:d2c}, we discuss how to topologize $\mathbb{I}_n$ and related spaces.

In Section \ref{sec:geo}, we introduce a number of tools from discrete geometry to analyze the MSTD question for elements in $\mathbb{I}$. The main result of that section is the following.

\begin{theorem}\label{thm:no_mstd}
For $n \leq 3$, there does not exist $\mathcal{A} \in \mathbb{I}_n$ such that $\mathcal{A}$ is MSTD. For all $n \geq 4$, there do exist MSTD $\mathcal{A} \in \mathbb{I}_n$.
\end{theorem}

This theorem may be loosely interpreted as the continuous analogue of the theorem of Hegarty \cite{hegarty} that there are no MSTD subsets of $\mathbb{Z}$ of cardinality less than 8. 

In addition to allowing us to prove Theorem \ref{thm:no_mstd}, the tools developed in Section \ref{sec:geo} will give us a way of producing an infinite parametrized family of MSTD subsets of $\mathbb{Z}$ from a single MSTD set (either in the discrete or continuous sense). This infinite family will in fact have a simple algebraic structure. As is to be seen, our techniques in some sense allow one to uncover the structure of the set which resulted in it being MSTD and then to systematically enumerate all MSTD sets with this same structure. This is progress towards answering the open-ended question of Nathanson in \cite{nathanson2}: ``What is the structure of finite sets satisfying $|A+A| > |A-A|$?'' These ideas are presented in Section \ref{sec:one2many}.

In Section \ref{sec:c2d} we present a sort of converse result to Theorem \ref{basic:thm}, namely that up to affine transformation, given any $\mathcal{A} \in \mathbb{I}$, we can find an $A \subset \mathbb{Z}$ such that the additive behavior of $\mathcal{A}$ and $A$ are as similar as we like. Finally, in Section \ref{sec:conclusion}, we present some experimental data and pose some open questions and lines of further research.

\section{Discrete to Continuous}\label{sec:d2c}

In the sequel, $A$ always denotes a finite subset of $\mathbb{Z}$. For convenience, in this section we assume that elements of $\mathbb{I}$ consist of closed intervals.

\begin{definition}
Let $a, b \in \mathbb{Z}$ with $a \leq b$. We call the set $[a, b]_\mathbb{Z} := \{a \leq x \leq b | x \in \mathbb{Z}\}$ a \textbf{closed integer interval}.
\end{definition}

\begin{definition}
Let $A = \{a_1, \dots, a_n\}$ with $a_1 < a_2 < \dots < a_n$. The \textbf{interval decomposition} of $A$ is the unique decomposition of $A$ into closed integer intervals $A = [b_1, c_1]_\mathbb{Z} \cup [b_2, c_2]_\mathbb{Z} \cup \dots \cup [b_k, c_k]_\mathbb{Z}$ such that for all $i \neq j$, we have $|b_i - c_j| \geq 2$ (that is, adjacent integers are always grouped into the same closed integer interval; see Example \ref{ex:int_decomp}).
\end{definition}

\begin{example}\label{ex:int_decomp}
Let $A = \{0, 1, 3, 4, 5, 7, 9, 10\}$. Then the interval decomposition of $A$ is 
\begin{gather}
A\ =\ [0, 1]_\mathbb{Z} \cup [3, 5]_\mathbb{Z} \cup [7, 7]_\mathbb{Z} \cup [9, 10]_\mathbb{Z}.
\end{gather}
\end{example}

\begin{definition}
Suppose $A$ has interval decomposition $A = [b_1, c_1]_\mathbb{Z} \cup \dots \cup [b_k, c_k]_\mathbb{Z}$. The \textbf{continuous representation} of $A$, denoted $A^*$, is
\begin{gather}
A^*\ :=\ \left[b_1 - \frac{1}{4}, c_1 + \frac{1}{4}\right] \cup \dots \cup \left[b_k - \frac{1}{4}, c_k + \frac{1}{4}\right].
\end{gather}
\end{definition}

We are now ready to state our main theorem of this section.

\begin{theorem} \label{discrete to continuous}
Let $A$ and $B$ be finite subsets of $\mathbb{Z}$, and $A^*$ and $B^*$ their continuous representations. Let $\mu$ be the Lebesgue measure on the real line. Then it is true that
\begin{gather}
|A + B|\ =\ \mu(A^* + B^*)
\end{gather}
and
\begin{gather}
|A - B|\ =\ \mu(A^* - B^*).
\end{gather}
\end{theorem}

Before proving this theorem we shall prove a sequence of ancillary propositions. The following proposition is a straightforward exercise.

\begin{proposition} \label{one interval}
Let $A = [a_1, a_2]_\mathbb{Z}$ and $B = [b_1, b_2]_\mathbb{Z}$ with $a_1 < b_1$. Then, the following are true:
\begin{gather}
A+B\ =\ [a_1+b_1, a_2+b_2]_\mathbb{Z}, \\
A^* + B^*\ =\ \left[a_1 + b_1 - \frac{1}{2}, a_2 + b_2 + \frac{1}{2}\right], \\
A - B\ =\ [a_1 - b_2, a_2 - b_1]_\mathbb{Z},
\end{gather}
and
\begin{gather}
A^* - B^*\ =\ \left[a_1 - b_2 - \frac{1}{2}, a_2 - b_1 + \frac{1}{2}\right].
\end{gather}
\end{proposition}

\begin{corollary}
If $A = [a_1, a_2]_\mathbb{Z}$ and $B = [b_1, b_2]_\mathbb{Z}$, then $|A + B| = \mu(A^* + B^*)$ and $|A - B| = \mu(A^* - B^*)$.
\end{corollary}

\begin{proposition} \label{which n}
Let $A$ and $B$ be finite subsets of $\mathbb{Z}$. Let $n$ be in $\mathbb{Z}$. Then $n$ is in $A+B$ if and only if $n$ is in $A^* + B^*$. Similarly, $n$ is in $A-B$ if and only if $n$ is in $A^* - B^*$. 
\end{proposition}

\begin{proof}
This follows almost immediately from Proposition \ref{one interval}. Let $A = [a_1, b_1]_\mathbb{Z} \cup \dots \cup [a_k, b_k]_\mathbb{Z}$ and $B = [c_1, d_1]_\mathbb{Z} \cup \dots \cup [c_\ell, d_\ell]_\mathbb{Z}$ be the interval decompositions of $A$ and $B$ respectively. Suppose $n \in A + B$. This clearly can only happen if $n \in [a_i, b_i]_\mathbb{Z} + [c_j, d_j]_\mathbb{Z}$ for some $i$ and $j$. By Proposition \ref{one interval}, letting $A' = [a_i, b_i]_\mathbb{Z}$ and $B' = [c_j, d_j]_\mathbb{Z}$, we know that $n \in A'^* + B'^*$ and therefore $n$ is also in $A^* + B^*$. 

Now suppose that $n \in A^* + B^*$ with $n \in \mathbb{Z}$. This implies that $n \in [a_i, b_i]_\mathbb{Z}^* + [c_j, d_j]_\mathbb{Z}^*$ for some $i$ and $j$. By Proposition \ref{one interval}, this implies that $n \in [a_i, b_i]_\mathbb{Z} + [c_j, d_j]_\mathbb{Z}$ since $([a_i, b_i]_\mathbb{Z}^* + [c_j, d_j]_\mathbb{Z}^*) \cap \mathbb{Z} = [a_i, b_i]_\mathbb{Z} + [c_j, d_j]_\mathbb{Z}$. Thus we get that $n \in A+B$.

By switching the plus signs above to minus signs, we obtain a proof for the second half of the proposition statement.
\end{proof}

\begin{proposition} \label{cover}
Let $A$ and $B$ be finite subsets of $\mathbb{Z}$. Let $C_{A^* + B^*}$ and $C_{A^* - B^*}$ be defined as
\begin{gather}
C_{A^* + B^*}\ :=\ \bigcup_{n \in (A^* + B^*) \cap \mathbb{Z}} \overline{B}_{\frac{1}{2}}(n) 
\end{gather}
and
\begin{gather}
C^*_{A^* - B^*}\ :=\ \bigcup_{n \in (A^* - B^*) \cap \mathbb{Z}} \overline{B}_{\frac{1}{2}}(n).
\end{gather}
where $\overline{B}_{\frac{1}{2}}(n)$ is the closed ball of radius $\frac{1}{2}$ centered at $n$. Then $C_{A^* + B^*} = A^* + B^*$ and $C_{A^* - B^*} = A^* - B^*$. 
\end{proposition}

\begin{proof}
Notice that this statement is clearly true in the case that $A = [a_1, a_2]_\mathbb{Z}$ and $B = [b_1, b_2]_\mathbb{Z}$ from Proposition \ref{one interval}. Suppose now that $A = \bigcup_{i = 1}^k I_i$ and $B = \bigcup_{j = 1}^\ell J_j$ where the $I_i$ are the integer intervals in the interval decomposition of $A$, and the $J_j$ are the integer intervals in the interval decomposition of $B$. Thus
\begin{equation}
\begin{split}
C_{A^* + B^*} & \ =\ \bigcup_{1 \leq i \leq k, 1 \leq j \leq \ell} C_{I_i^* + J_j^*} \\
	&\ =\  \bigcup_{1 \leq i \leq k, 1 \leq j \leq \ell} I_i^* + J_j^* \\
	&\ =\ A^* + B^*. 
\end{split}
\end{equation}
As before, by changing plus signs to minus signs we get a proof for the latter part of the proposition statement.
\end{proof}

We are now ready to prove Theorem \ref{discrete to continuous}.

\begin{proof}[Proof of Theorem \ref{discrete to continuous}]
By Proposition \ref{cover}, we know that $C_{A^* + B^*} = A^* + B^*$. The set $C_{A^* + B^*}$ is composed to sets of measure 1 such that the intersection of any pair of them is either empty or a single point. This implies that the measure of the intersection of any pair of these sets is 0. We can therefore conclude that $\mu(A^* + B^*) = \#\{(A^* + B^*) \cap \mathbb{Z}\}$. However, by Proposition \ref{which n}, we know that $\#\{(A^* + B^*) \cap \mathbb{Z}\} = |A + B|$. Therefore, $\mu(A^* + B^*) = |A + B|$. Showing that $\mu(A^* - B^*) = |A - B|$ follows analogously.
\end{proof}

The above results show that by studying sumsets and difference sets for collections of intervals, we can retrieve results about collections of intervals as a special case. However, the power of instead studying collections of intervals is that, as we continuously vary the endpoints of our intervals, the size of the sumset and of the difference set also vary continuously. Therefore, for example, given a single MSTD collection of intervals, we can vary this the endpoints of these intervals slightly and still have an MSTD set.

In general, rather than deal with $\mathbb{I}$, we shall fix some $n$ and deal with $\mathbb{I}_n$. Since additive behavior (in particular the property of being MSTD) is invariant under affine transformation, modding out by affine equivalence does not change $\mathbb{I}_n$ in a meaningful way. With this in mind, there are several natural ways to topologize $\mathbb{I}_n$ and its quotient by some or all of the affine group. These fall into two broad categories of parametrizations: \textbf{endpoint parametrizations} and \textbf{interval-gap parametrizations}.

Endpoint parametrizations refer to subsets of some Euclidean space where each component of a vector in the space is either a left or right endpoint for some interval on the real line. The \textbf{free simplex model} is the subset of $\mathbb{R}^{2n}$ composed of vectors of the form $(a_1, b_1, \dots, a_n, b_n)$ with the condition that $a_1 \leq b_1 \leq \dots \leq a_n \leq b_n$. We think of this vector as representing the set $[a_1, b_1] \cup [a_2, b_2] \cup \dots \cup [a_n, b_n]$. This model is a parametrization of all of $\mathbb{I}_n$; we have not modded out by any affine equivalences. This model will be particularly useful in the next section.

One disadvantage of the above model is that the space is not compact. A similar model, which we call the \textbf{simplex model} is the subset of $\mathbb{R}^{2n}$ composed of vectors $(a_1, b_1, \dots, a_n, b_n)$ with the condition that $0 \leq a_1 \leq b_1 \leq \dots \leq b_1 \leq b_2 \leq 1$. This model is named as such because the points in this space all live in a $2n$-dimensional simplex.

Yet another disadvantage of both of the above models is that there is some redundancy: up to affine transformation we still have several representatives for the same set. We can mod out by all affine transformations (with positive determinant) by requiring that our leftmost interval start at 0 and our rightmost interval end at 1. We thus define the \textbf{restricted simplex model} to be those vectors $(b_1, a_2, b_2, \dots, a_{n-1}, b_{n-1}, a_n)$ in $\mathbb{R}^{2n-2}$ such that $0 \leq b_1 \leq a_2 \dots \leq b_{n-1} \leq a_n \leq 1$ with the understanding that this vector corresponds to the collection of intervals $[0, b_1] \cup [a_2, b_2] \cup \dots \cup [a_n, 1]$.

Another class of natural parametrizations we call interval-gap parametrizations; in these, instead of designating where intervals begin and end, we designate how long each interval is and how long the gaps between consecutive intervals are (thus we have already modded out by translations). The \textbf{free cube model} consists of vectors $(\ell_1, g_1, \dots, g_{n-1}, \ell_n)$ in $\mathbb{R}^{2n-1}$ such that $\ell_i,g_i \geq 0$ for all $i$. Given such a vector, we construct a collection of $n$ intervals as follows: the leftmost interval is $[0, \ell_1]$. The gap between the leftmost interval and its neighboring interval to the right is $g_1$ and the length of this next interval is $\ell_2$, and so on. 

The above model has a natural compactification which we call the \textbf{unit cube model} in which we require that $0 \leq \ell_i, g_i \leq 1$ for all $i$ (so that we may think of our points as living in the unit cube in $\mathbb{R}^{2n-1}$). Though we do not use this model in this paper, this model was utilized in \cite{ballot_polytope} to give a geometric interpretation to the enumeration and growth rate of bidirectional ballot sequences/$(1, 1)$-culminating paths (these sequences were used by Zhao in his constructions of MSTD sets).

Lastly, analogously to the restricted simplex model, we can mod out by all affine transformations to get the \textbf{restricted unit cube model}. That is, we can additionally require that $\ell_1 + g_1 + \dots + \ell_n = 1$. This model is essentially the same as the restricted simplex model.

\section{A Geometric Perspective}\label{sec:geo}

The main goal of this section is to prove Theorem \ref{thm:no_mstd}. However, as we shall see, in doing so we shall develop a powerful set of tools for analyzing continuous sumsets and difference sets. These tools will end up being useful for studying subsets of $\mathbb{Z}$ as well.

In the sequel we shall generally let $n$ be fixed, and let $J$ represent some element in $\mathbb{I}_n$ consisting of closed intervals. Thus, $J = \{J_1, \dots, J_n\}$ where $J_i = [x_i, y_i]$, and for $i < j$, $J_i$ is to the left of $J_j$ on the number line. Vectors will be denoted by parentheses (i.e., $[x, y]$ is an interval and $(x, y)$ is a vector). 

To handle the $n=1$ case of Theorem \ref{thm:no_mstd} is more or less trivial. Already with the $n=2$ case some work is required; the analysis of the $n=2$ case reveals most of the important ideas that go into the $n \geq 3$ case, and we choose to analyze the $n=2$ case in such a way that the core ideas are clearest, rather than using the more powerful but harder to visualize framework used in the $n \geq 3$ case.

\begin{lemma}
There are no MSTD sets in $\mathbb{I}_1$.
\end{lemma}

\begin{proof}
Suppose $J = [x_1, y_1]$. The sumset consists of a single interval, $[x_1 + x_1, y_1 + y_1]$, which has length $2 y_1 - 2 x_1$. The difference set also consists of a single interval, $[x_1 - y_1, y_1 - x_1]$, which has length $2 y_1 - 2 x_1$. Thus, the length of the sumset is exactly equal to the length of the difference set, so in all cases $J$ is balanced. Notice that this analysis reveals that there is only one type of behavior in the $n=1$ case, which is to be expected since all intervals are equivalent mod affine transformations.
\end{proof}


The $n=2$ case requires some more work. As has been previously discussed, $J+J$ can be expressed as the union over all $i$ and $j$ of intervals of the form $J_i + J_j$ (and similarly for the difference set). If we know the locations of all intervals of the form $J_i + J_j$ (or $J_i - J_j$) relative to each other, then we know exactly what the sumset (or difference set) is. More precisely, if we know the total ordering on the left and right endpoints of all intervals of the form $J_i + J_j$ (or $J_i - J_j$), then we know exactly which points are in $J+J$ (or $J-J$), and additionally, we know the measure of $J+J$. This motivates the following definition.

\begin{definition}
Given $J$, the total ordering on the left and right endpoints of intervals of the form $J_i + J_j$ ($J_i - J_j$) is called the \textbf{structure} of the sumset (difference set). The structure of $J$ refers to the structure of both the sumset and the difference set.
\end{definition}

Thus, stated succinctly, the above observations say that if we know the structure of the sumset/difference set, then we know exactly what the sumset/difference set is.

Another notational definition which will make analysis easier is the following:
\begin{definition}
Let $(J_i \pm J_j)_L$ and $(J_i \pm J_j)_R$ denote the left and right endpoints respectively of the interval $J_i \pm J_j$. 
\end{definition}

\begin{lemma}
There are no MSTD sets in $\mathbb{I}_2$
\end{lemma}

\begin{proof}
Let $J = J_1 \cup J_2$. If we use the free simplex model, then $J$ has four degrees of freedom. However, if we instead use the restricted simplex model, then $J$ only has two degrees of freedom, so all possible cases for the structure of $J$ can be readily visualized. Using this model, we can represent a given $J$ as a point in $\mathbb{R}^2$: if $J = \{[0, y_1], [x_2, 1]\}$, we represent this as the point $\overline{J} = (y_1, x_2) \in \mathbb{R}^2$. The restricted simplex model restricts to the region in the plane simultaneously satisfying the following inequalities:
\begin{gather}
y_1 \geq 0, \label{ieq1}\\
x_2 \leq 1, \\
y_1 \leq x_2. \label{ieq2}
\end{gather}
We call this region $A$.

First we see what information we need to figure out the structure of $J+J$. There are three intervals we are concerned with: $J_1 + J_1$, $J_1 + J_2$ and $J_2 + J_2$. Note that since $J_1$ is to the left of $J_2$, we immediately know that:
\begin{gather}
(J_1 + J_1)_L\ \leq\ (J_1 + J_2)_L\ \leq\ (J_2 + J_2)_L, \nonumber\\
(J_1 + J_1)_R\ \leq\ (J_1 + J_2)_R\ \leq\ (J_2 + J_2)_R, \nonumber\\
(J_1 + J_1)_L\ \leq\ (J_1 + J_1)_R, \nonumber\\
(J_1 + J_2)_L\ \leq\ (J_1 + J_2)_R, \nonumber\\
(J_2 + J_2)_L\ \leq\ (J_2 + J_2)_R.
\end{gather}

Thus to figure out the structure of $J+J$, we only need the following information:
\begin{gather}
(J_1+J_1)_R \stackrel{?}{\ \leq\ } (J_1+J_2)_L \nonumber\\
(J_1+J_2)_R \stackrel{?}{\ \leq\ } (J_2+J_2)_L.
\end{gather}

These two inequalities in question are equivalent to knowing on which side of the following lines in the plane our point $\overline{J}$ lies:
\begin{gather}
y_1 + y_1\ =\ 0 + x_2 \label{line1} \\
y_1 + 1\ =\ x_2 + x_2.  \label{line2}
\end{gather}

We now turn to determining the structure of $J-J$. From equations \eqref{ieq1} to \eqref{ieq2}, we already know:
\begin{gather}
(J_1 - J_2)_L\ \leq\ (J_1 - J_1)_L\ \leq\ (J_1 - J_1)_R\ \leq\ (J_2 - J_1)_R \nonumber\\
(J_1 - J_2)_L\ \leq\ (J_2 - J_2)_L\ \leq\ (J_2 - J_2)_R\ \leq\ (J_2 - J_1)_R \nonumber\\
(J_1 - J_2)_L\ \leq\ (J_1 - J_2)_R\ \leq\ (J_2 - J_1)_L\ \leq\ (J_2 - J_1)_R.
\end{gather}

By the symmetry of the difference set, we only need to know the following information to determine the structure of the difference set:
\begin{gather}
(J_1 - J_1)_R \stackrel{?}{\ \leq\ } (J_2 - J_2)_R \nonumber\\
(J_1 - J_1)_R \stackrel{?}{\ \leq\ } (J_2 - J_1)_L \nonumber\\
(J_2 - J_2)_L \stackrel{?}{\ \leq\ } (J_2 - J_1)_L.
\end{gather}

These three inequalities in question are equivalent to knowing on which side of the following lines in the plane our point $\overline{J}$ lies:
\begin{gather}
y_1 - 0\ =\ 1 - x_2 \label{line3}\\
y_1 - 0\ =\ x_2 - y_1 \label{line4}\\
1 - x_2\ =\ x_2 - y_1. \label{line5}
\end{gather}

However notice that equation \ref{line4} is the same as equation \ref{line1}, and also \ref{line5} is the same as \ref{line2}. Thus for points within $A$, by knowing on which side of the three lines given by equations \ref{line1}, \ref{line2}, and \ref{line3} our point $\overline{J}$ lies, we completely know the structure of $J$. All of the above information is captured geometrically in Figure \ref{fig:n2case}. We see that $A$ is divided into 6 regions such that all points in the same region have the same structure. Table \ref{table:huge} at the end of this paper enumerates which structure each of these regions corresponds to.

\begin{figure}
\begin{tikzpicture} [scale=7]
\draw [ultra thick] (0, 0) to (0, 1);
\draw [ultra thick] (0, 0) to (1, 1);
\draw [ultra thick] (0, 1) to (1, 1);
\draw [->] (0, 0) to (0, 1.2);
\draw [->] (0, 0) to (1.2, 0);
\draw (0, 0) to (0.5, 1);
\draw (0, 0.5) to (1, 1);
\draw (0, 1) to (0.5, 0.5);
\draw [->] (0.75, 1.08) to (0.75, 1.01);
\node [above] at (0.75, 1.08) {$x_2 = 1$};
\node [below] at (0.3, 0) {$y_1$ axis };
\node [left] at (0, 0.3) {$x_2$ axis};
\draw [->] (-0.08, 0.75) to (-0.01, 0.75);
\node [left] at (-0.08, 0.75) {$y_1 = 0$};
\draw [<-] (0.76, 0.74) to (0.8, 0.7);
\node [right] at (0.79, 0.69) {$x_2 = y_1$};
\draw [dotted] (0.5, 1) to (0.6, 1.2);
\draw [dotted] (0, 0) to (-0.1, -0.2);
\draw [dotted] (0, 1) to (-0.15, 1.15);
\draw [dotted] (0.5, 0.5) to (0.65, 0.35);
\draw [dotted] (0, 0.5) to (-0.2, 0.4);
\draw [dotted] (1, 1) to (1.2, 1.1);
\draw [<-] (0.54, 1.1) to (0.46, 1.1);
\node [left] at (0.46, 1.1) {$2 y_1 = x_2$};
\draw [<-] (1.1, 1.04) to (1.1, 0.96);
\node [below] at (1.1, 0.96) {$1 + y_1 = 2 x_2$};
\draw [<-] (0.57, 0.44) to (0.64, 0.44);
\node [right] at (0.64, 0.44) {$1 - x_2 = y_1$};
\draw (0.3, 0.9) circle [radius = 0.04];
\node  at (0.3, 0.9) {1};
\draw (0.1, 0.75) circle [radius = 0.04];
\node  at (0.1, 0.75) {2};
\draw (0.6, 0.9) circle [radius = 0.04];
\node  at (0.6, 0.9) {3};
\draw (0.1, 0.4) circle [radius = 0.04];
\node  at (0.1, 0.4) {4};
\draw (0.5, 0.65) circle [radius = 0.04];
\node  at (0.5, 0.65) {5};
\draw (0.35, 0.5) circle [radius = 0.04];
\node  at (0.35, 0.5) {6};
\end{tikzpicture}
\caption{The space $A$ is partitioned into six regions such that within each region the structure is constant. Note that all regions are defined by a system of linear inequalities}
\label{fig:n2case}
\end{figure}
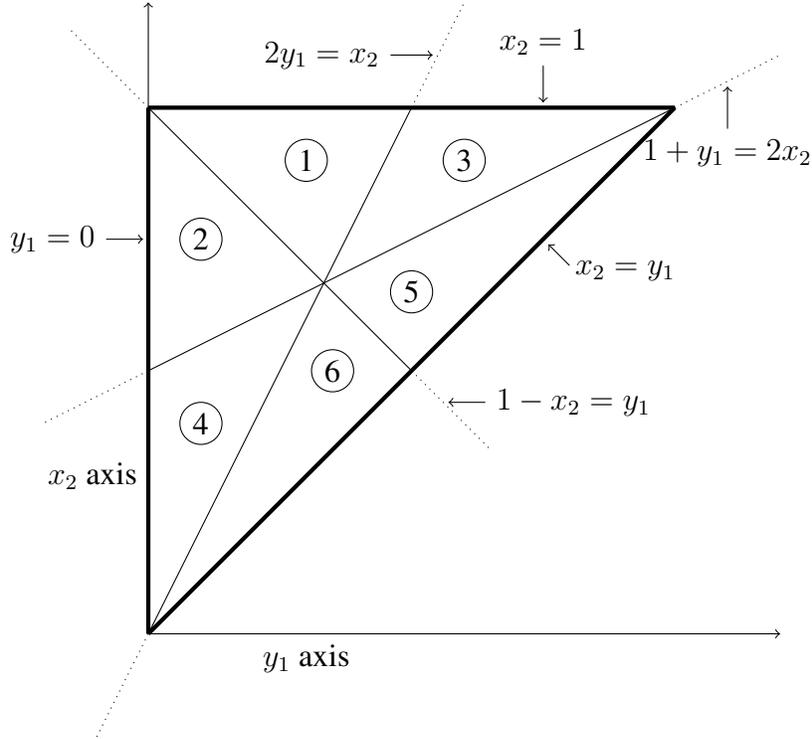

In Table \ref{table:huge} we claim that regions 1-4 are difference dominant, and regions 5 and 6 are balanced. That regions 3 and 4 are difference dominant and that regions 5 and 6 are balanced are immediate to see. To see that regions 1 and 2 are difference dominant is also straightforward. We shall show this is the case for region 1. The proof for region 2 is similar.

In (the interior of) region 1, the following inequalities hold:
\begin{gather}
2 y_1\ <\ x_2 \nonumber\\
1 + y_1\ <\ 2 x_2 \nonumber\\
1 - x_2\ <\ y_1. \label{ieq3}
\end{gather}
From Table \ref{table:huge} we see that $\mu (J+J) = 3 y_1 - 3 x_2 + 3$ and $\mu (J-J) = 4 y_1 - 2 x_2 + 2$. Therefore, $\mu (J-J) - \mu (J+J) = y_1 + x_2 - 1$. We thus are interested in the following inequality:
\begin{gather}
0 \stackrel{?}{\ <\ } y_1 + x_2 - 1.
\end{gather}
However from equation \eqref{ieq3} we know that this inequality holds true everywhere in region 1. Therefore, region 1 is difference dominant.
\end{proof}

In handling the $n = 3$ case of Theorem \ref{thm:no_mstd}, rather than use the restricted simplex model, we shall use the free simplex model. The main utility of this model is that the sum of two vectors in $\mathbb{I}_n$, as represented in this model, is again in $\mathbb{I}_n$ in this model (that is, $\mathbb{I}_n$ has a semigroup structure). Similarly to before, if $J = \{[x_1, y_1], \dots, [x_n, y_n]\}$, we associate to this the point $\overline{J} = (x_1, y_1, \dots, x_n, y_n)$. The goal is again to determine the structure of $J$. Like in the $n=2$ case, in order to figure out the total ordering on the left and right endpoints of the intervals in the sumset/difference set, it suffices to know the outcomes of every comparison between endpoints. Each such comparison can be expressed as evaluation of a linear polynomial in the $x_i$s and $y_i$s. After a bit of thought, one realizes that up to multiplication by unit, there are four types of such polynomials that we are interested in, as described in Table \ref{table:polys}. We call such polynomials \textbf{comparison polynomials}.

\begin{figure}[h]
\begin{center}
\begin{tabular}{ l  l }
 Type of polynomial & Purpose \\
\hline
$0 \stackrel{?}{<} x_i + x_j - x_k - x_\ell$ & 
\begin{tabular}{l}
Compares $(J_i + J_j)_L$ and $(J_k+J_\ell)_L$ 
\end{tabular} \\
$0 \stackrel{?}{<} x_i + x_j - y_k - y_\ell$ & 
\begin{tabular}{l}
 \\
Compares $(J_i + J_j)_L$ and $(J_k + J_\ell)_R$ \\
If $i < k$ and $\ell < j$, compares $(J_i - J_k)_L$ and $(J_\ell - J_j)_L$
\end{tabular} \\
$0 \stackrel{?}{<} y_i + y_j - y_k - y_\ell$ & 
\begin{tabular}{l}
Compares $(J_i + J_j)_R$ and $(J_k+J_\ell)_R$ 
\end{tabular} \\
$0 \stackrel{?}{<} x_i + y_\ell - x_j - y_k$ & 
\begin{tabular}{l}
If $i < k$ and $\ell < j$, compares $(J_i - J_k)_L$ and $(J_\ell - J_j)_L$
\end{tabular}
\newline \\
\newline
\end{tabular}
\end{center}
\captionof{table}[Table]{The four types of comparison polynomials, along with exactly what type of comparison each type of polynomial is good for.}
\label{table:polys}
\end{figure}

For each linear polynomial as in Table \ref{table:polys}, we have an associated \textbf{coefficient vector}, $v$, namely the vector which when dotted with $(x_1, y_1, \dots, x_n, y_n)$ gives the linear polynomial in question. Additionally, to each such linear polynomial we have an \textbf{associated hyperplane} $H$, namely the hyperplane obtained by setting the polynomial equal to zero. This hyperplane partitions $\mathbb{R}^{2n}$ into two pieces corresponding to the two difference outcomes of the comparison which the linear polynomial represents. Note that $v$ is a basis for the orthogonal complement to $H$.

We now recall some definitions from discrete geometry.

\begin{definition}
A \textbf{central hyperplane arrangement} is a finite collection of hyperplanes which all pass through the origin.
\end{definition}

Note that the set of associated hyperplanes to the linear polynomials of interest form a central hyperplane arrangment which we call the \textbf{structure arrangement}.

\begin{definition}
A \textbf{conical combination} of vectors $v_1$, $\dots$, $v_m$ is any combination of the form
\begin{gather}
\alpha_1 v_1 + \dots + \alpha_m v_m
\end{gather}
such that $\alpha_i \geq 0$ for all $1 \leq i \leq m$. The set of all conical combinations of a set of vectors is called the \textbf{conical span}.
\end{definition}

\begin{definition}
A \textbf{polyhedral cone} is the conical span of a fixed finite set of vectors. Equivalently, it is all points in the intersection of finitely many halfspaces for which the corresponding set of hyperplanes forms a central hyperplane arrangement.
\end{definition}

In this paper we shall be dealing exclusively with rational polyhedral cones, so we can always assume that each generator of the cone is a vector in $\mathbb{Z}^n$ with relatively prime entries (a primitive integer vector).

\begin{definition}
An \textbf{orientation} on a hyperplane is a choice to label one of the corresponding halfspaces as positive and the other as negative. Equivalently, it is a choice of a non-zero vector $b$ in the (1D) orthogonal complement to the hyperplane: the positive halfspace is the set of points $v$ such that $v \cdot b \geq 0$. We call $b$ a \textbf{positive normal}.
\end{definition}

If we have a central hyperplane arrangement, then for each way of simultaneously orienting all the hyperplanes in the arrangement, we get an associated polyhedral cone (this cone may just be zero). Our space is thus partitioned into disjoint polyhedral cones such that disjoint cones have disjoint interiors. Dually, if we choose a specific polyhedral cone arising from a central hyperplane arrangement, then we get an induced orientation on the hyperplane arrangement as follows: for each hyperplane, we say that the halfspace containing the polyhedral cone is the positive halfspace. 

Given the structure arrangement, we get a partition of $\mathbb{R}^{2n}$ into finitely many polyhedral cones, each of which we call a \textbf{structure cone}. For a given structure cone, when choosing positive normals for the induced orientation on the arrangement, we may choose for each hyperplane either the corresponding coefficient vector or its negative.

\begin{definition}
Let $V$ be a polyhedral cone in $\mathbb{R}^n$. Let $V^* = \{w \in \mathbb{R}^{n} : \forall \ v \in V,\ w \cdot v \geq 0 \}$. Then $V^*$ is called the \textbf{dual cone} of $V$.
\end{definition}

We may interpret the dual cone to $V$ as the set of all possible positive normals to oriented hyperplanes such that all of $V$ lies on the positive side of the hyperplane (together with the zero vector).

The following basic result from the theory of polyhedral cones makes it easy to find dual cones, especially when the polyhedral cones are given in terms of intersections of half-spaces (as is the case here).

\begin{proposition}\label{farkas}
Given a polyhedral cone, $V$, the conical span of the set of positive normals in the induced orientation on $V$ is the dual cone of $V$.
\end{proposition}

Suppose $V$ is a structure cone. For all $J$ such that $\overline{J} \in V$, the structure of $J$ is the same. For all $\overline{J} \in V$, there exists a single homogeneous linear polynomial in the $x_i$'s and $y_i$'s, call it $P_+$ with coefficient vector $v_+$, such that $\mu(J+J) = v_+ \cdot \overline{J}$. Analogously there exists $P_-$ with coefficient vector $v_-$ such that $\mu(J-J) = v_- \cdot \overline{J}$. The vectors $v_+$ and $v_-$ are called the \textbf{sum vector} and \textbf{difference vector}, respectively, for the cone $V$. The vector $d = v_+ - v_-$ is called the \textbf{MSTD vector} for $V$; a collection of intervals $J \in \mathbb{I}_n$ with $\overline{J} \in V$ is an MSTD set if and only if $d \cdot \overline{J} > 0$. Note that if $d \cdot \overline{J} = 0$, then $J$ is balanced, and if $d \cdot \overline{J} < 0$, then $J$ is difference dominant.

The MSTD vector gives rise to an oriented hyperplane, namely the orthogonal complement to the span of $d$, with positive normal equal to $d$. We thus have the following crucial observation.

\begin{observation}\label{observation:crucial}
Let $V$ be a structure cone with MSTD vector $d$. If and only if $d = 0$, then all of $V$ is balanced. Now suppose $d$ is non-zero. Then, if and only if $d$ is in the dual cone to $V$, then $V$ is sum-dominant (except possibly on the boundary which may be balanced). If and only if $-d$ is in the dual cone to $V$, then $V$ is difference domaint (except possibly on the boundary which may be balanced). In all other cases, $V$ splits into a sum dominant region and a difference dominant region.
\end{observation}

\begin{lemma}
For all $J \in \mathbb{I}_3$, $J$ is not an MSTD set.
\end{lemma}

\begin{proof}
We now have all the tools necessary to handle the $n = 3$ case of Theorem \ref{thm:no_mstd}. In light of Observation \ref{observation:crucial}, rather than needing  to (somehow) check the uncountably many possible collections of three intervals, we instead now need only check that for each structure cone $V$ for the $\mathbb{I}_3$ structure arrangement, either $d = 0$ or $-d \in V^*$. To prove the $n = 3$ case, we proceed as in the $n = 2$ case, namely we enumerate all structure cones and show that in all cases either $d = 0$ or $-d \in V^*$. We present an algorithm to carry our this procedure. A variant on this algorithm was implemented in SAGE \cite{sagemath} on a computer to enumerate and check all cases. Whereas there were only six cases for $n = 2$, there are over 500 cases for $n = 3$.

The difficult part of the algorithm is enumerating all of the structure cones. We shall describe a simple algorithm for doing just that; after stating the algorithm we shall discuss exactly what it does (and why it works).

\begin{algorithm}[h]
\LinesNumbered
 \textbf{Input}:{\ the number of intervals, $n$} \\
 \textbf{Output}:{\ the dual cones to the structure cones} \\
 \emph{partial\_dual\_cones} = $\{ \{ \hat{y}_1 - \hat{x}_1, \hat{x}_2 - \hat{y}_1, \dots, \hat{y}_n - \hat{x}_n\} \}$ \\
 \emph{list\_of\_normals} = \emph{GENERATE\_LIST\_OF\_NORMALS} ($n$) \\
 \For{new\_normal in list\_of\_normals}{
 	\emph{new\_partial\_dual\_cones} = $\emptyset$ \\
	\For{partial\_dual\_cone in partial\_dual\_cones}{
		\If{IS\_CONSISTENT (partial\_dual\_cone, new\_normal)}{
			\emph{new\_partial\_dual\_cones} $\cup =$ (\emph{partial\_dual\_cone} $\cup$ \emph{new\_normal}) 
		}
		\If{IS\_CONSISTENT (partial\_dual\_cone, $-$new\_normal}{
			\emph{new\_partial\_dual\_cones} $\cup =$ (\emph{partial\_dual\_cone} $\cup$ ($-$\emph{new\_normal}))
		}
	
	}
	\emph{partial\_dual\_cones} = \emph{new\_partial\_dual\_cones}
 
 }
\Return{partial\_dual\_cones}
\medskip
\caption{Algorithm to find the structure cones with non-empty interior.}
\label{alg:mstd}
\end{algorithm}

In essence, Algorithm \ref{alg:mstd} adds one hyperplane at a time to our arrangement, keeping track at each step what the non-trivial cones are (the cones that aren't just the zero vector). It represents each cone by the generators of its dual cone, that is by the choice of orientation (positive normals) on the partial hyperplane arrangement which gives rise to that specific cone. 

Since for all $J \in \mathbb{I}_n$, $x_1 \leq y_1 \leq \dots \leq x_n \leq y_n$, for every structure cone of interest to us, the vectors $\hat{y}_1 - \hat{x}_1, \dots, \hat{y}_n - \hat{x}_n$ must be positive normals (where $\hat{x}_i$ is the coefficient vector to the polynomial $x_i$). This explains line 3. The variable \emph{partial\_dual\_cones} keeps track of all partial choices of positive normals for structure cones in our arrangement.

In line 4 of Algorithm 1, the function \emph{GENERATE\_LIST\_OF\_NORMALS} returns the coefficient vectors for all the relevant comparison polynomials (irredundantly up to multiplication by unit).

The remainder of the algorithm consists of two loops. The outer loop iterates through the set of normals in this list, and the inner loop iterates through the cones in our partial hyperplane arrangement. A single iteration of the outer loop introduces a new hyperplane to the arrangement (as represented by its normal \emph{new\_normal}) and a single iteration of the inner loop keeps track if a given cone in the partial arrangement (as represented by \emph{partial\_dual\_cone}) splits into two cones. The function \emph{IS\_CONSISTENT}(\emph{partial\_dual\_cone}, \emph{new\_normal}) tests whether or not \emph{$-$new\_normal} is in the conical span of the elements of \emph{partial\_dual\_cone}. In other words, it tests if the cone corresponding to \emph{partial\_dual\_cone} lies entirely on the negative side of the oriented hyperplane with positive normal \emph{new\_normal}. If so, then adding \emph{new\_normal} to \emph{partial\_dual\_cone} as a positive normal would result in a cone with empty interior. Any such cone would be on the boundary of some other cone with non-empty interior, and thus is safe to ignore since it would be handled by other cases. An example of an inconsistent choice of positive normal is the following: suppose we already know that $x_1 \leq x_2 \leq x_3$. Then we necessarily know that $x_1 + x_2 \leq x_2 + x_3$, and thus adding the vector $\hat{x}_1 + \hat{x}_2 - \hat{x}_2 - \hat{x}_3$ to our list of positive normals is inconsistent. Testing consistency in this sense can be phrased a a feasibility problem in linear programming, and thus there exist efficient algorithms to solve this problem. The variable \emph{new\_partial\_dual\_cones} keeps track of the new cones in our new hyperplane arrangement with non-empty interior. Once the interior loop has finished, we set \emph{partial\_dual\_cones} equal to \emph{new\_partial\_dual\_cones} and repeat the outer loop, until we finally finish and return the set of cones with non-empty interior in our structure hyperplane arrangement.

Once we have all of our structure cones, as represented by the the list of positive normals for the induced orientation on the hyperplane arrangement, we have all of the information we need to figure out the sum vector and difference vector for each cone. We do not explicity describe an algorithm to do so here, but leave it as an exercise to the interested reader to think about how to do this.

Once we have the sum and difference vectors, we have the MSTD vector. Testing if the MSTD vector is in the dual cone is a cone membership problem and thus can also be solved by linear programming techniques.

We implemented in SAGE a slight variation on Algorithm \ref{alg:mstd}. In particular, our algorithm does not enumerate all regions where the structure is constant (i.e. all structure cones), but rather just those regions such that within a region, the sum vector and difference vector is constant (e.g. in the $n = 2$ case, regions 5 and 6 have the same sum formula and difference formula, but have different structures; using an improvement on Algorithm \ref{alg:mstd}, these two regions would not be distinguished). We do not describe the improved algorithm here; the mere existence of an algorithm such as Algorithm \ref{alg:mstd} is what is most important. When the improved algorithm ran, it found 502 cases, and in all such cases either $d = 0$ or $-d$ was in the dual cone implying there are no sum dominant sets for $n = 3$. On a personal computer with 4GB of RAM, the computation took around 30 minutes to complete.
\end{proof}

\begin{proof}[Proof of Theorem \ref{thm:no_mstd}]
To complete the proof of Theorem \ref{thm:no_mstd}, we must for that for $n \geq 4$, there do exists collections of $n$ intervals which are sum dominant. Note that it suffices to just show that this is the case when $n = 4$ (if this is not clear, see Section \ref{sec:conclusion} on cleaving). We can actually turn to the existing literature on MSTD sets of integers to find an example. In Hegarty \cite{hegarty}, we have
\begin{gather}
\{0, 1, 2, 4, 5, 9, 12, 13, 14\}.
\end{gather}
The corresponding collection of intervals is:
\begin{gather}
J \ =\  \left\{\left[-\frac{1}{4}, 2+\frac{1}{4}\right], \left[4-\frac{1}{4}, 5+\frac{1}{4}\right], \left[9-\frac{1}{4}, 9+\frac{1}{4}\right], \left[12-\frac{1}{4}, 14+\frac{1}{4}\right]\right\}.
\end{gather}
\end{proof}

\section{From one MSTD set to many}\label{sec:one2many}
In this section we show a method of producing a large class of MSTD sets (both continuous and discrete) from a single set. In particular this gives a new method of producing parametrized families of MSTD sets of integers. 

The idea of the method is straightforward given the ideas in Section \ref{sec:geo}: given $J \in \mathbb{I}_n$, we find the generators for its structure cone. We then also find the MSTD vector, $d$, for this cone. Since $J$ is MSTD, we know that either $d \in V^*$, implying the entire cone is MSTD, or else the oriented hyperplane with positive normal $d$ partitions the hyperplane into two cones, both the non-empty interior. One of these cones is entirely MSTD. We then find the generators for this cone.

Before formally presenting the details of the algorithm (Algorithm \ref{alg:cones2}), we show the procedure on a concrete example. This example has a few peculiarities which make the example simpler than most, but the core ideas are present.

\begin{example}
Let $G = \{0, 1, 2, 4, 5, 9, 12, 13, 14\}$ and let $J$ be its continuous representation. A computer program reveals that $J$ is on the boundary of 108 different structure cones. We choose one of these cones arbitrarily and call it $V$. The extremal rays of $V$ are the columns of the following matrix:
\begin{gather}
\begin{bmatrix}
0 & 0 & 0 & 0 & 0 & 0 & 0 & 1 & -1 \\
1 & 1 & 2 & 2 & 1 & 1 & 1 & 1 & -1 \\
2 & 2 & 4 & 3 & 2 & 2 & 2 & 1 & -1 \\
3 & 2 & 6 & 4 & 3 & 3 & 3 & 1 & -1 \\
4 & 4 & 10 & 7 & 5 & 5 & 6 & 1 & -1 \\
5 & 4 & 10 & 7 & 5 & 5 & 6 & 1 & -1 \\
6 & 5 & 13 & 9 & 7 & 7 & 8 & 1 & -1 \\
7 & 7 & 16 & 11 & 8 & 9 & 10 & 1 & -1
\end{bmatrix}.
\end{gather}
We then compute the MSTD vector, $d$, for V:
\begin{gather}
d\ =\ 
\begin{pmatrix}
-1 & 2 & -2 & 0 & 1 & 2 & -2 & 0
\end{pmatrix}^T.
\end{gather}
We check if $d \in V^*$ and find that it is not. We therefore add $d$ to $V^*$ and then take its dual cone to get a new cone $W$. The extremal rays for $W$ are the columns of the following matrix:
\begin{gather}
A\ =\ \begin{bmatrix}
0   &   0   &   0   &   0   &   0   &   0   &   0   &   1   &   -1 \\
1   &   1   &   2   &   3   &   2   &   1   &   3   &   1   &   -1 \\
2   &   2   &   4   &   5   &   3   &   2   &   5   &   1   &   -1 \\
3   &   2   &   6   &   7   &   4   &   3   &   7   &   1   &   -1 \\
4   &   4   &   10  &   12  &   7   &   6   &   12  &   1   &   -1 \\
5   &   4   &   10  &   12  &   7   &   6   &   12  &   1   &   -1 \\
6   &   5   &   13  &   16  &   9   &   8   &   16  &   1   &   -1 \\
7   &   6   &   16  &   19  &   11  &   10  &   20  &   1   &   -1
\end{bmatrix}.
\end{gather}
Note then that any vector in the conical column span of $A$ is either balanced or sum dominant. A closer examination reveals that all the columns give rise to balanced sets except for the 5th row. Therefore, any vector of the form $A z$ with $z_i \geq 0$ and $z_5 > 0$ gives rise to a continuous MSTD set. Thus from finding a single MSTD set, we have found a huge class of continuous MSTD sets.

With just a bit more work we can find an infinite family of MSTD sets of integers as well. In fact, we shall find all MSTD subsets of $\mathbb{Z}$ whose continuous representation corresponds to a point in $W$. Recall that given a point corresponding to an MSTD set of the form $a = (x_1 - 1/4, y_1 + 1/4, \dots, x_n - 1/4, y_n + 1/4)$ where each $x_i$ and $y_i$ is in $\mathbb{Z}$, then we can obtain an MSTD set of integers, namely $\{[x_1, y_1]_\mathbb{Z}, \dots, [x_n, y_n]_\mathbb{Z}\}$. If $a$ is in some cone $X$, then $\alpha a \in X$ for all $\alpha \geq 0$. In particular, $4 a \in W$; note that $4 a \in \mathbb{Z}^{2n}$. Furthermore, mod 4, the entries of $4 a$ must be $(3, 1, 3, 1, \dots, 3, 1)$. Conversely, if some point in $\mathbb{Z}^{2n}$ is of the form $(3, 1, \dots, 3, 1)$ mod 4, then we can find an MSTD set of integers from that point by dividing by 4 and undoing the discrete to continuous process. We say that a point $x \in \mathbb{R}^{2n}$ is an \textbf{$(\alpha, \beta)$-lattice point mod $k$} if $x \in \mathbb{Z}^{2n}$ and mod $k$, $x = (\alpha, \beta, \dots, \alpha, \beta)$. Thus our goal is to find the $(3, 1)$-lattice points mod $4$ in the conical column span of $A$.

Let $B$ be the matrix $A$ without the last column (the last two columns are linearly dependent, so right now we can ignore the last column). From the above observations, we are interested in finding solutions to the following system of equations:
\begin{gather}
B \alpha\ =\ 
\begin{bmatrix}
0 & 0 & 0 & 0 & 0 & 0 & 0 & 1 \\
 1 & 1 & 2 & 3 & 2 & 1 & 3 & 1  \\
 2 & 2 & 4 & 5 & 3 & 2 & 5 & 1  \\
 3 & 2 & 6 & 7 & 4 & 3 & 7 & 1  \\
 4 & 4 & 10 & 12 & 7 & 6 & 12 & 1 \\
 5 & 4 & 10 & 12 & 7 & 6 & 12 & 1 \\
 6 & 5 & 13 & 16 & 9 & 8 & 16 & 1 \\
 7 & 6 & 16 & 19 & 11 & 10 & 20 & 1
\end{bmatrix}
\begin{bmatrix}
\alpha_1 \\
\alpha_2 \\
\alpha_3 \\
\alpha_4 \\
\alpha_5 \\
\alpha_6 \\
\alpha_7 \\
\alpha_8
\end{bmatrix}
\ =\
\begin{bmatrix}
1 \\
3 \\
1 \\
3 \\
1 \\
3 \\
1 \\
3 
\end{bmatrix}
\pmod{4}.
\end{gather}
Using Mathematica \cite{math}, we get that the unique such $\alpha$ is
\begin{gather}
\alpha\ =\ 
\begin{pmatrix}
2 & 0 & 0 & 0 & 0 & 0 & 0 & 3
\end{pmatrix}^T.
\end{gather}
Therefore, by dividing any vector of the form below by 4 and then undoing the discrete to continuous process, we obtain an MSTD set of integers (see Table \ref{table:examples} for some examples):
\begin{gather}
\begin{bmatrix}
0 & 0 & 0 & 0 & 0 & 0 & 0 & 1 \\
 1 & 1 & 2 & 3 & 2 & 1 & 3 & 1  \\
 2 & 2 & 4 & 5 & 3 & 2 & 5 & 1  \\
 3 & 2 & 6 & 7 & 4 & 3 & 7 & 1  \\
 4 & 4 & 10 & 12 & 7 & 6 & 12 & 1 \\
 5 & 4 & 10 & 12 & 7 & 6 & 12 & 1 \\
 6 & 5 & 13 & 16 & 9 & 8 & 16 & 1 \\
 7 & 6 & 16 & 19 & 11 & 10 & 20 & 1
\end{bmatrix}
\begin{bmatrix}
2 + 4\beta_1 \\
4\beta_2 \\
4\beta_3 \\
4\beta_4 \\
4 + 4\beta_5 \\
4\beta_6 \\
4\beta_7 \\
3 + 4\beta_8
\end{bmatrix}
\ \ \ \
\begin{matrix}
\beta_1 \in \mathbb{N}_0 \\
\beta_2 \in \mathbb{N}_0 \\
\beta_3 \in \mathbb{N}_0 \\
\beta_4 \in \mathbb{N}_0 \\
\beta_5 \in \mathbb{N}_0 \\
\beta_6 \in \mathbb{N}_0 \\
\beta_7 \in \mathbb{N}_0 \\
\beta_8 \in \mathbb{Z} \\
\end{matrix}.
\end{gather}
Furthermore, since the $B$ has 8 rows, which is the dimension of the vector space $W$ lives in, and since the determinant of $B$ is -1, we know that in fact all MSTD sets of integers in $W$ are obtained in this way (in general it will not be the case that the number of generators is equal to the dimension of the space, or even if the number of generators is equal to the dimension, that the determinant of the matrix whose rows are those generators will be $\pm 1$; we shall discuss how to deal with these issues shortly). 

In the above procedure, we took a single MSTD set and not only found a huge class of continuous MSTD sets, but also a non-trivial infinite family of discrete MSTD sets (in fact all of the discrete MSTD sets in the structure cone). Furthermore, the MSTD set we started with led us to find 108 different MSTD cones, so from the same starting MSTD set, we can carry out the above process 107 more times to find even more MSTD sets (all of this arising from finding a single MSTD set)!
\end{example}

\begin{center}
\begin{table}[b]
\begin{tabular}{c c c c c c c c c}
$\beta_1$ & $\beta_2$ & $\beta_3$ & $\beta_4$ & $\beta_5$ & $\beta_6$ & $\beta_7$ & $\beta_8$ & MSTD integer set\\
\hline
0	&	0	&	0	&	0	&	0	&	0	&	0	&	-1	&	$\{0, 1, 2, 4, 5, 9, 12, 13, 14\}$ \\
1	&	0	&	0	&	0	&	0	&	0	&	0	&	-1 	&	$\{0, 1, 2, 3, 6, 7, 8, 13, 14, 18, 19, 20, 21\}$ \\
0	&	1	&	0	&	0	&	0	&	0	&	0	&	-1 	&	$\{0, 1, 2, 3, 6, 7, 13, 17, 18, 19, 20\}$ \\
0	&	0	&	1	&	0	&	0	&	0	&	0	&	-1 	&	$\{0, 1, 2, 3, 4, 8, 9, 10, 11, 19, 25, 26, 27, 28, 29, 30\}$ \\
0	&	0	&	0	&	1	&	0	&	0	&	0	&	-1  &	$\{0, 1, 2, 3, 4, 5, 9, 10, 11, 12, 21, 28, 29, 30, 31, 32, 33\}$ \\
0	&	0	&	0	&	0	&	1	&	0	&	0	&	-1 	&	$\{0, 1, 2, 3, 4, 7, 8, 9, 16, 21, 22, 23, 24, 25\}$ \\
0	&	0	&	0	&	0	&	0	&	1	&	0	&	-1 	&	$\{0, 1, 2, 3, 6, 7, 8, 15, 20, 21, 22, 23, 24, 25\}$ \\
0	&	0	&	0	&	0	&	0	&	0	&	1	&	-1 	&	$\{0, 1, 2, 3, 4, 5, 9, 10, 11, 12, 21, 28, 29, 30, 31, 32, 33, 34\}$
\end{tabular}
\caption{A few of the MSTD sets of integers contained in this structure cone.}
\label{table:examples}
\end{table}
\end{center}

We now discuss a more general algorithm for carrying out the above procedure. There are three main issues which up to this point we have glossed over. 
\begin{enumerate}
    \item The set $J$ may be contained in multiple structure cones, so we need a way of enumerating all structure cones containing $J$; this can be resolved using ideas similar to those presented in Algorithm \ref{alg:mstd}.
    \item A given structure cone containing $J$ may have more generators than the dimension of the space; the trick here is to partition the cone into a collection of simplicial cones.
    \item For a given rational simplicial cone with generators represented as primitive integer vectors, if the determinant of the corresponding matrix is not $\pm 1$, then integer conical combinations of the generators will not necessarily give all lattice points in the cone (and in particular will not necessarily give all lattice points corresponding to MSTD subsets of $\mathbb{Z}$).
\end{enumerate}

First we discuss issue (1); Algorithm \ref{alg:cones2} gives a way of resolving this this issue. In words, this algorithm first orients those hyperplanes for which $\overline{J}$ is on the strictly positive side. This results in a polyhedral cone for which $\overline{J}$ is in the interior. The remaining hyperplanes (as represented by \emph{non\_strict\_normals}) are then one by one tested to see if they partition the partial cones found so far into two cones with non-empty interiors. Lines 15-28 are virtually identical to Algorithm \ref{alg:mstd} and thus their function should be clear.

%
\begin{algorithm}[!ht]
\LinesNumbered
 \textbf{Input}:{\ a point $\overline{J}$ corresponding to an MSTD set consisting of $n$ intervals} \\
 \textbf{Output}:{\ all structure cones containing $\overline{J}$} \\
 \emph{list\_of\_normals} = \emph{GENERATE\_LIST\_OF\_NORMALS}($n$) \\
 \emph{partial\_cone} = $\emptyset$ \\
 \emph{non\_strict\_normals} = $\emptyset$ \\
 \For{new\_normal in list\_of\_normals}{
    \emph{dotted} = \emph{new\_normal} $\cdot$ $\overline{J}$ \\
    \uIf{dotted > 0}{
        \emph{partial\_cone} $\cup = $ \emph{new\_normal}
    }
    \uElseIf{dotted < 0}{
        \emph{partial\_cone} $\cup = $ $-$\emph{new\_normal}
    }
    \uElse{
        \emph{non\_strict\_normals} $\cup = $ \emph{new\_normal}
    }
 }
 \emph{partial\_cones} = $\{$\emph{partial\_cone}$\}$ \\
 \For{normal in non\_strict\_normal}{
    \emph{new\_partial\_cones} = $\emptyset$ \\
    \For{cone in partial\_cones}{
        \If{IS\_CONSISTENT(cone, normal)}{
            \emph{new\_partial\_cones} $\cup =$ (\emph{cone} $\cup$ \emph{normal})
        }
        \If{IS\_CONSISTENT(cone, $-$normal)}{
            \emph{new\_partial\_cones} $\cup =$ (\emph{cone} $\cup$ $-$\emph{normal})
        }
    }
    \emph{partial\_cones} = \emph{new\_partial\_cones} \\
 }
\Return{partial\_cones}
\medskip
\caption{Algorithm describing how to find all the structure cones containing a given point.}
\label{alg:cones2}
\end{algorithm}

Once we have the representation of a polyhedral cone as an intersection of halfspaces, there are algorithms to find its representation as the concial span of a collection of generators. Issue (2) is then quite straightforward to deal with. Partitioning a polyhedral cone into simplicial cones is virtually the same as partitioning a compact polytope into simplices and there exist algorithms to do so.

Issue (3) is also not too bad to deal with. Let $V$ be a rational simplicial cone. Let $A$ be a matrix whose columns are primitive integer vectors generating the cone. Then, $\det(A) \in \mathbb{Z}$. Let $D = |\det(A)|$. Then, since $A^{-1} = \textnormal{adj}(A)/\det(A)$, all $x$ such that $Ax \in \mathbb{Z}$ are in $\mathbb{Z}^{n}/D$, that is the set of points whose product with $D$ is in $\mathbb{Z}^n$. Suppose $x \in \mathbb{Z}^n/D$ such that $Ax$ is a $(3, 1)$-lattice point mod $4$. Then, $A (Dx)$ must be a $(3D, D)$-lattice point mod $4D$. Conversely, if $A y$ is a $(3D, D)$-lattice point mod $4D$, then $A (y/D)$ is a $(3, 1)$-lattice point mod $4$. Therefore, finding all $x \in \mathbb{Q}^n$ such that $A x$ is a $(3, 1)$-lattice point mod $4$ is equivalent to finding all $y \in \mathbb{Z}^n$ such that $A y$ is a $(3D, D)$-lattice point mod $4D$. To do this we need only find the set of solutions to the following system of equations over $\mathbb{Z}/(4 D \mathbb{Z})$.
\begin{gather}
A y\ =\ \begin{bmatrix}
3D \\
D \\
\vdots \\
3D \\
D
\end{bmatrix} \pmod{4D} \label{eqn:mod4d}.
\end{gather}
Suppose that $v_1, \dots, v_{2n}$ are the generators for some simplicial cone, $S$, arising as the MSTD refinement of an MSTD structure cone with MSTD vector $d$. Let $v_{i_1}, \dots, v_{i_k}$ be those $v_i$ such that $v_i \cdot d > 0$. Let $p$ be any particular solution to \ref{eqn:mod4d}. A point $m$ is an MSTD $(3, 1)$-lattice point mod $4$ if and only if it can be expressed as
\begin{gather}
m \ =\ A\bigg(\frac{p}{D} + \frac{k}{D} + 4\ell\bigg)
\end{gather}
where $k$ is in the kernel of $A$ as an endomorphism on $(\mathbb{Z}/(4D)\mathbb{Z})^{2n}$, and $\ell \in \mathbb{Z}^{2n}$ and such that $f = p/D + k/D + 4\ell$ satisfies $f \cdot \hat{e}_i \geq 0$ for all $i \in [2n]$ and $f \cdot \hat{e}_{i_j} > 0$ for some $j \in [k]$.

\section{Continuous to Discrete}\label{sec:c2d}
In this short section we prove a simple ``converse'' to Theorem \ref{basic:thm}: up to scaling, every element in $\mathbb{I}$ can be arbitrarily well approximated by (the continuous representation of) a finite collection of integers.

\begin{theorem}\label{thm:c2d}
Let $\mathcal{A} \in \mathbb{I}$. For every $\varepsilon > 0$, there exists $\alpha > 0$ and $B \subset \mathbb{Z}$, with continuous representation $\mathcal{B}$, such that
$$ \mu \left( (\alpha \mathcal{A} + \alpha \mathcal{A}) \ \Delta \ (\mathcal{B} + \mathcal{B})\right)\ <\ \varepsilon,$$
and
$$ \mu \left( (\alpha \mathcal{A} - \alpha \mathcal{A}) \ \Delta \ (\mathcal{B} - \mathcal{B})\right)\ <\ \varepsilon.$$
\end{theorem}

\begin{proof}
The idea of the proof of Theorem \ref{thm:c2d} is to dilate the set $\mathcal{A}$ and then approximate each dilated interval by the set of integers contained in the interval. Without loss of generality, we may assume that $\mathcal{A} \subset [0, 1]$. Suppose that $\mathcal{A}$ consists of $k$ intervals, that is $\mathcal{A} = J_1 \cup \dots \cup J_k$ with $J_i = [x_i, y_i]$ and with $J_i$ to the left of $J_j$ for $i < j$. Suppose the length of the shortest of these intervals is $\delta$. Let $N \in \mathbb{Z}$ be any number such that
\begin{gather}
N\ \geq\ \max \left(\frac{3}{\delta}, \frac{8 k^2}{\varepsilon}\right). \label{eqn:N}
\end{gather}
Let $F_N = \{i/N : 0 \leq i \leq N, \ i \in \mathbb{Z}\}$. By equation \eqref{eqn:N}, we know
\begin{gather}
\# (J_i \cap F_N)\ \geq\ 3 \label{eqn:geq3}
\end{gather}
Let $\ell_i, r_i \in \mathbb{Z}$ be such that $\ell_i/N = \min (J_i \cap F_N)$ and $r_i/N = \max (J_i \cap F_N)$. Notice that by equation \eqref{eqn:N}, 
\begin{gather}
\left|\frac{\ell_i+1}{N} - x_i \right|\ <\ \frac{2}{N} \label{eqn:left}
\end{gather}
and
\begin{gather}
\left|y_i - \frac{r_i-1}{N} \right|\ <\ \frac{2}{N}. \label{eqn:right}
\end{gather}
Let $B_i = [\ell_i + 1, r_i - 1]_{\mathbb{Z}}$ (by equation \ref{eqn:geq3} each $B_i$ is non-empty). Let $B = \bigcup_i B_i$. Let $\mathcal{B}$ be the continuous representation of $B$ with $\mathcal{B}_i$ the continuous representation of $B_i$. Let $\mathcal{C}$ be the set $\mathcal{B}$ scaled by $1/N$, and $\mathcal{C}_i$ the set $\mathcal{B}_i$ scaled by $1/N$. Notice that $\mathcal{C} \subseteq \mathcal{A}$. Therefore
\begin{gather}
\mathcal{D}\ :=\ (\mathcal{A} \pm \mathcal{A}) \ \Delta \ (\mathcal{C} \pm \mathcal{C})\ =\ (\mathcal{A} \pm \mathcal{A}) \setminus (\mathcal{C} \pm \mathcal{C}).
\end{gather}
Let $x \in \mathcal{D}$. We have that
\begin{gather}
\mathcal{D}\ \subseteq\ \bigcup_{i, j} \left((J_i \pm J_j) \setminus (C_i \pm C_j)\right)
\end{gather}
Therefore
\begin{equation}
\begin{split}
\mu (\mathcal{D}) &\ \leq\ \sum_{i, j} \mu \left((J_i \pm J_j) \setminus (C_i \pm C_j)\right) \\
	&\ \leq\ k^2 \max_{i, j} \mu \left((J_i \pm J_j) \setminus (C_i \pm C_j)\right).
\end{split}
\end{equation}
By equations \eqref{eqn:left} and \eqref{eqn:right}, we know that
\begin{gather}
\max_{i, j} \mu \left((J_i \pm J_j) \setminus (C_i \pm C_j)\right)\ <\ \frac{8}{N}.
\end{gather}
Therefore,
\begin{equation}
\begin{split}
	\mu(\mathcal{D}) \ <\ \frac{8 k^2}{N} \ \leq\ \varepsilon.
\end{split}
\end{equation}
\end{proof}

\section{Open questions and concluding remarks}\label{sec:conclusion}
The ideas presented in Sections \ref{sec:geo} and \ref{sec:one2many} motivate several interesting follow-up questions. First, there's the question of whether or not a more elegant proof of Theorem \ref{thm:no_mstd} exists.

\begin{question}
Is there a proof of Theorem \ref{thm:no_mstd} that does not reduce to casework?
\end{question}

There are also several interesting combinatorial questions that arise. One basic question is:
\begin{question}\label{q:how_many}
How many cones are there in the structure hyperplane arrangement for $\mathbb{I}_n$?
\end{question}

A closely related question has been investigated before in \cite{golomb}. The number of such regions is closely related to (and upper bounded by) the even indexed entries in OEIS A237749. There is a rich theory of counting the number of regions in a hyperplane arrangement (see \cite{stanley}, e.g.), and perhaps these techniques could answer Question \ref{q:how_many}.

Another basic question is:
\begin{question}
How many MSTD structure cones are there for $\mathbb{I}_n$? What are their relative (intrinsic) volumes? Is there a ``dominating'' MSTD cone?
\end{question}

In our opinion, one of the most interesting subsequent question is the following.
\begin{question}\label{q:connected}
Do the set of MSTD points in $\mathbb{I}_n$ form a connection region? If so, what is the degree of connectivity of this region (is it $2n$-connected?)? If not, how many connected components does it contain? Does the number of connected components change as $n$ increases?
\end{question}

If the answer to Question \ref{q:connected} is yes, then it would in some sense imply that there is only one ``type'' of MSTD set, from a single MSTD set (with a fixed number of intervals), all other MSTD sets can be found by perturbing that set (and keeping it MSTD along the way).

Given an MSTD $J \in \mathbb{I}_n$, there are several ways of naturally ``embedding'' this set into $\mathbb{I}_{n+1}$. If $J$ is composed of open intervals, then removing any single point in $J$ results in $n+1$ intervals, call it $J'$, but the sets $J$ and $J'$ are basically the same. We say that $J'$ is obtained from $J$ by \textbf{cleaving}. Assuming the answer to Question \ref{q:connected} is no, a refined question is

\begin{question}
Can every MSTD point in $\mathbb{I}_{n+1}$ be obtained by an MSTD path from the image of some cleaved MSTD point in $\mathbb{I}_n$?
\end{question}

If we deal with a compact parameter space, as in the simplex model and unit cube model, we may then talk about the probability that a point is MSTD, balanced, or difference dominant. Figures \ref{graph:balanced} and \ref{graph:mstd} show approximations of these probabilities based on Monte Carlo simulation with 10 million trials.

\begin{center}
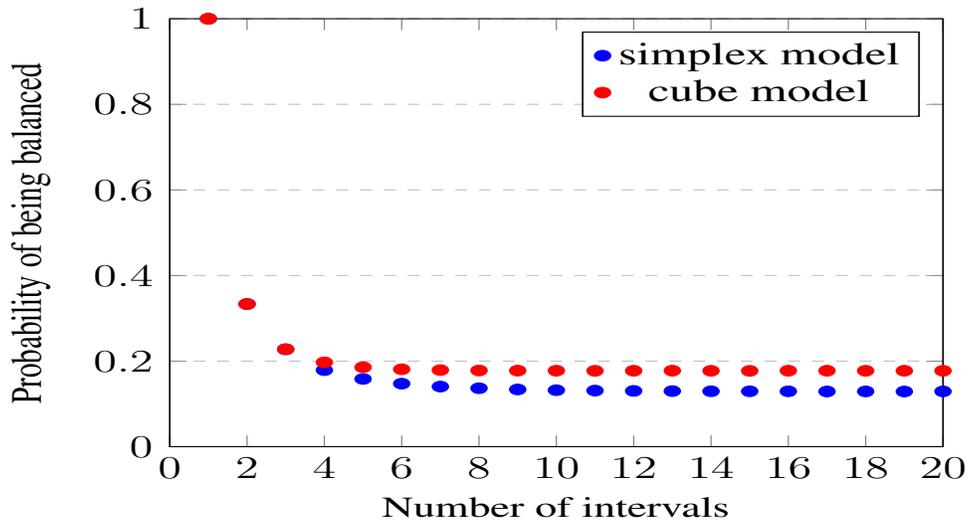
\begin{figure}[!h]
\begin{tikzpicture}[xscale=1.5]
\begin{axis}[
    xlabel={\small Number of intervals},
    ylabel style = {align=left}, ylabel={\small Probability of being balanced},
    xmin=0, xmax=20,
    ymin=0, ymax=1,
    xtick = {0, 2, 4, 6, 8, 10, 12, 14, 16, 18, 20},
    ytick = {0, 0.2, 0.4, 0.6, 0.8, 1.0},
    legend pos=north east,
    ymajorgrids=true,
    grid style=dashed,
]
 
\addplot[
    color=blue,
    only marks,
    ]
    coordinates {(1, 1) (2, 0.3333170) (3, 0.2275130) (4, 0.17887) (5, 0.158289) (6, 0.1472604) (7, 0.1404734) (8, 0.1365882) (9, 0.1337495) (10, 0.1320153) (11, 0.1310538) (12, 0.1303472) (13, 0.1299088) (14, 0.1295455) (15, 0.1294703) (16, 0.1293930) (17, 0.1290355) (18, 0.129145) (19, 0.1289232) (20, 0.1291887)};

\addplot[
    color=red,
    only marks,
    ]
    coordinates {(1, 1) (2, 0.3332818) (3, 0.2278143) (4, 0.1971369) (5, 0.1856548) (6, 0.1808756) (7, 0.1789348) (8, 0.1780362) (9, 0.1777582) (10, 0.177595) (11, 0.1772671) (12, 0.1774295) (13, 0.1776763) (14, 0.1773738) (15, 0.1771929) (16, 0.1775619) (17, 0.1775284) (18, 0.1773592) (19, 0.1774349) (20, 0.1774034)};
    
    \addlegendentry{simplex model}
    \addlegendentry{cube model}
 
\end{axis}
\end{tikzpicture}
\caption{Probability of being balanced in the simplex and cube models based on Monte Carlo simulation (10 million trials).}
\label{graph:balanced}
\end{figure}
\end{center}

\begin{center}
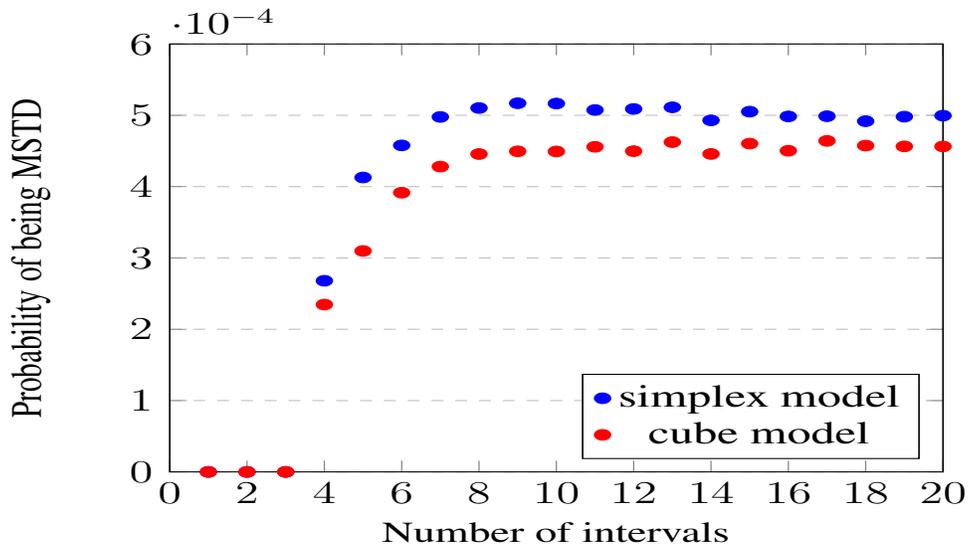
\begin{figure}[!h]
\begin{tikzpicture}[xscale=1.5]
\begin{axis}[
    xlabel={\small Number of intervals},
    ylabel style = {align=left}, ylabel={\small Probability of being MSTD},
    xmin=0, xmax=20,
    ymin=0, ymax=0.0006,
    xtick = {0, 2, 4, 6, 8, 10, 12, 14, 16, 18, 20},
    ytick = {0, 0.0001, 0.0002, 0.0003, 0.0004, 0.0005, 0.0006},
    legend pos=south east,
    ymajorgrids=true,
    grid style=dashed,
]
 
\addplot[
    color=blue,
    only marks,
    ]
    coordinates {(1, 0) (2, 0) (3, 0) (4, 0.0002680) (5, 0.0004128) (6, 0.0004578) (7, 0.0004977) (8, 0.0005102) (9, 0.0005169) (10, 0.0005165) (11, 0.0005074) (12, 0.0005089) (13, 0.0005112) (14, 0.0004929) (15, 0.0005051) (16, 0.0004984) (17, 0.0004987) (18, 0.0004917) (19, 0.0004981) (20, 0.0004995)};

\addplot[
    color=red,
    only marks,
    ]
    coordinates {(1, 0) (2, 0) (3, 0) (4, 0.0002346) (5, 0.0003098) (6, 0.0003915) (7, 0.0004281) (8, 0.0004457) (9, 0.0004495) (10, 0.0004493) (11, 0.0004558) (12, 0.0004497) (13, 0.0004624) (14, 0.0004458) (15, 0.0004603) (16, 0.0004503) (17, 0.0004641) (18, 0.0004575) (19, 0.0004564) (20, 0.0004563)};
    
    \addlegendentry{simplex model}
    \addlegendentry{cube model}
 
\end{axis}
\end{tikzpicture}
\caption{Probability of being MSTD in the simplex and cube models based on Monte Carlo simulation (10 million trials).}
\label{graph:mstd}
\end{figure}
\end{center}

Interestingly, the probabilities for being MSTD and balanced appear to be different for the simplex model and unit cube model. However, in both cases, the probability of being sum-dominant appears to converge to a similar value to the limiting probability in the discrete case ($\sim 4.5 \times 10^{-4}$).

\begin{question}
Do the probabilities of being sum dominant and balanced converge for the simplex model and the cube model? What is the relationship between these MSTD probabilities and the limiting MSTD probability in the discrete case?
\end{question}

One of the main open questions in the study of MSTD sets is to construct a constant density family of MSTD sets as $n \to \infty$. Thus we may ask:

\begin{question}
Can the techniques in this paper be used to construct a constant density family of MSTD subsets of $[n]$ as $n \to \infty$?
\end{question}

There are several interesting subsequent lines of inquiry stemming from the ideas in the paper. More generally, we believe that there is a lot of utility in passing from the discrete to the continuous as in this paper. Ideas closely related to those here were utilized in the related paper \cite{ballot_polytope} to reveal a geometric structure to a certain family of combinatorial objects which was not visible previously. We believe there may be several further fruitful applications of the ideas of this paper.

\newgeometry{margin=0.7in}
\begin{landscape}
\thispagestyle{plain}
\begin{figure}[ht]
\begin{tikzpicture}[xscale = 22, yscale=16.5]
\draw (0, 1) to (1, 1);
\draw (0, 1) to (0, 0);
\draw (0, 0) to (1, 0);
\draw (1, 0) to (1, 1);
\draw (0, 1-2.0/17.0) to (1, 1-2.0/17.0);
\draw (1/13, 1) to (1/13, 0);
\draw (3/13, 1) to (3/13, 0);
\draw (5/13, 1) to (5/13, 0);
\draw (7/13, 1) to (7/13, 0);
\draw (9/13, 1) to (9/13, 0);
\draw (11/13, 1) to (11/13, 0);
\node at (1.0/26.0, 1-2.5/17.0) {\large $J_1 + J_1$};
\node at (1.0/26.0, 1-3.5/17.0) {\large $J_1 + J_2$};
\node at (1/26, 1-4.5/17) {\large $J_2 + J_2$};
\draw (0, 1-5/17) to (1, 1-5/17);
\node at (1/26, 1-5.5/17) {\large $\mathbf{J + J}$};
\draw (0, 1-6/17) to (1, 1-6/17);
\node at (1/26, 1-6.5/17) {\large $J_1 - J_1$};
\node at (1/26, 1-7.5/17) {\large $J_1 - J_2$};
\node at (1/26, 1-8.5/17) {\large $J_2 - J_1$};
\node at (1/26, 1-9.5/17) {\large $J_2 - J_2$};
\draw (0, 1-10/17) to (1, 1-10/17);
\node at (1/26, 1-10.5/17) {\large $\mathbf{J - J}$};
\draw (0, 1-11/17) to (1, 1-11/17);
\node at (1/26, 1-11.5/17) {\small $2 y_1 \stackrel{?}{\leq} x_2$};
\draw (0, 1-12/17) to (1, 1-12/17);
\node at (1/26, 1-12.5/17) {\tiny $1+y_1 \stackrel{?}{\leq} 2 x_2$};
\draw (0, 1-13/17) to (1, 1-13/17);
\node at (1/26, 1-13.5/17) {\tiny $1-x_2 \stackrel{?}{\leq} y_1$};
\draw (0, 1-14/17) to (1, 1-14/17);
\node at (1/26, 1-14.5/17) {\small $\mu(J+J)$};
\draw (0, 1-15/17) to (1, 1-15/17);
\node at (1/26, 1-15.5/17) {\small $\mu(J-J)$};
\draw (0, 1-16/17) to (1, 1-16/17);
\node at (1/26, 1-16.5/17) {\large Type};
\node at (2/13, 1-1/17) {\huge Region 1};
\node at (4/13, 1-1/17) {\huge Region 2};
\node at (6/13, 1-1/17) {\huge Region 3};
\node at (8/13, 1-1/17) {\huge Region 4};
\node at (10/13, 1-1/17) {\huge Region 5};
\node at (12/13, 1-1/17) {\huge Region 6};

\draw (0.084615, 0.852941) to (0.126154, 0.852941);
\draw (0.084615, 0.860294) to (0.084615, 0.845588);
\draw (0.126154, 0.860294) to (0.126154, 0.845588);
\draw (0.146923, 0.794118) to (0.174615, 0.794118);
\draw (0.146923, 0.801471) to (0.146923, 0.786765);
\draw (0.174615, 0.801471) to (0.174615, 0.786765);
\draw (0.209231, 0.735294) to (0.223077, 0.735294);
\draw (0.209231, 0.742647) to (0.209231, 0.727941);
\draw (0.223077, 0.742647) to (0.223077, 0.727941);

\draw (0.238462, 0.852941) to (0.252308, 0.852941);
\draw (0.238462, 0.860294) to (0.238462, 0.845588);
\draw (0.252308, 0.860294) to (0.252308, 0.845588);
\draw (0.290385, 0.794118) to (0.314615, 0.794118);
\draw (0.290385, 0.801471) to (0.290385, 0.786765);
\draw (0.314615, 0.801471) to (0.314615, 0.786765);
\draw (0.342308, 0.735294) to (0.376923, 0.735294);
\draw (0.342308, 0.742647) to (0.342308, 0.727941);
\draw (0.376923, 0.742647) to (0.376923, 0.727941);

\draw (0.392308, 0.852941) to (0.475385, 0.852941);
\draw (0.392308, 0.860294) to (0.392308, 0.845588);
\draw (0.475385, 0.860294) to (0.475385, 0.845588);
\draw (0.454615, 0.794118) to (0.503077, 0.794118);
\draw (0.454615, 0.801471) to (0.454615, 0.786765);
\draw (0.503077, 0.801471) to (0.503077, 0.786765);
\draw (0.516923, 0.735294) to (0.530769, 0.735294);
\draw (0.516923, 0.742647) to (0.516923, 0.727941);
\draw (0.530769, 0.742647) to (0.530769, 0.727941);

\draw (0.546154, 0.852941) to (0.560000, 0.852941);
\draw (0.546154, 0.860294) to (0.546154, 0.845588);
\draw (0.560000, 0.860294) to (0.560000, 0.845588);
\draw (0.573846, 0.794118) to (0.622308, 0.794118);
\draw (0.573846, 0.801471) to (0.573846, 0.786765);
\draw (0.622308, 0.801471) to (0.622308, 0.786765);
\draw (0.601538, 0.735294) to (0.684615, 0.735294);
\draw (0.601538, 0.742647) to (0.601538, 0.727941);
\draw (0.684615, 0.742647) to (0.684615, 0.727941);

\draw (0.700000, 0.852941) to (0.769231, 0.852941);
\draw (0.700000, 0.860294) to (0.700000, 0.845588);
\draw (0.769231, 0.860294) to (0.769231, 0.845588);
\draw (0.745000, 0.794118) to (0.803846, 0.794118);
\draw (0.745000, 0.801471) to (0.745000, 0.786765);
\draw (0.803846, 0.801471) to (0.803846, 0.786765);
\draw (0.790000, 0.735294) to (0.838462, 0.735294);
\draw (0.790000, 0.742647) to (0.790000, 0.727941);
\draw (0.838462, 0.742647) to (0.838462, 0.727941);

\draw (0.853846, 0.852941) to (0.902308, 0.852941);
\draw (0.853846, 0.860294) to (0.853846, 0.845588);
\draw (0.902308, 0.860294) to (0.902308, 0.845588);
\draw (0.888462, 0.794118) to (0.947308, 0.794118);
\draw (0.888462, 0.801471) to (0.888462, 0.786765);
\draw (0.947308, 0.801471) to (0.947308, 0.786765);
\draw (0.923077, 0.735294) to (0.992308, 0.735294);
\draw (0.923077, 0.742647) to (0.923077, 0.727941);
\draw (0.992308, 0.742647) to (0.992308, 0.727941);

\draw (0.133077, 0.617647) to (0.174615, 0.617647);
\draw (0.133077, 0.625000) to (0.133077, 0.610294);
\draw (0.174615, 0.625000) to (0.174615, 0.610294);
\draw (0.084615, 0.558824) to (0.112308, 0.558824);
\draw (0.084615, 0.566176) to (0.084615, 0.551471);
\draw (0.112308, 0.566176) to (0.112308, 0.551471);
\draw (0.195385, 0.500000) to (0.223077, 0.500000);
\draw (0.195385, 0.507353) to (0.195385, 0.492647);
\draw (0.223077, 0.507353) to (0.223077, 0.492647);
\draw (0.146923, 0.441176) to (0.160769, 0.441176);
\draw (0.146923, 0.448529) to (0.146923, 0.433824);
\draw (0.160769, 0.448529) to (0.160769, 0.433824);

\draw (0.300769, 0.617647) to (0.314615, 0.617647);
\draw (0.300769, 0.625000) to (0.300769, 0.610294);
\draw (0.314615, 0.625000) to (0.314615, 0.610294);
\draw (0.238462, 0.558824) to (0.262692, 0.558824);
\draw (0.238462, 0.566176) to (0.238462, 0.551471);
\draw (0.262692, 0.566176) to (0.262692, 0.551471);
\draw (0.352692, 0.500000) to (0.376923, 0.500000);
\draw (0.352692, 0.507353) to (0.352692, 0.492647);
\draw (0.376923, 0.507353) to (0.376923, 0.492647);
\draw (0.290385, 0.441176) to (0.325000, 0.441176);
\draw (0.290385, 0.448529) to (0.290385, 0.433824);
\draw (0.325000, 0.448529) to (0.325000, 0.433824);

\draw (0.420000, 0.617647) to (0.503077, 0.617647);
\draw (0.420000, 0.625000) to (0.420000, 0.610294);
\draw (0.503077, 0.625000) to (0.503077, 0.610294);
\draw (0.392308, 0.558824) to (0.440769, 0.558824);
\draw (0.392308, 0.566176) to (0.392308, 0.551471);
\draw (0.440769, 0.566176) to (0.440769, 0.551471);
\draw (0.482308, 0.500000) to (0.530769, 0.500000);
\draw (0.482308, 0.507353) to (0.482308, 0.492647);
\draw (0.530769, 0.507353) to (0.530769, 0.492647);
\draw (0.454615, 0.441176) to (0.468462, 0.441176);
\draw (0.454615, 0.448529) to (0.454615, 0.433824);
\draw (0.468462, 0.448529) to (0.468462, 0.433824);

\draw (0.608462, 0.617647) to (0.622308, 0.617647);
\draw (0.608462, 0.625000) to (0.608462, 0.610294);
\draw (0.622308, 0.625000) to (0.622308, 0.610294);
\draw (0.546154, 0.558824) to (0.594615, 0.558824);
\draw (0.546154, 0.566176) to (0.546154, 0.551471);
\draw (0.594615, 0.566176) to (0.594615, 0.551471);
\draw (0.636154, 0.500000) to (0.684615, 0.500000);
\draw (0.636154, 0.507353) to (0.636154, 0.492647);
\draw (0.684615, 0.507353) to (0.684615, 0.492647);
\draw (0.573846, 0.441176) to (0.656923, 0.441176);
\draw (0.573846, 0.448529) to (0.573846, 0.433824);
\draw (0.656923, 0.448529) to (0.656923, 0.433824);

\draw (0.734615, 0.617647) to (0.803846, 0.617647);
\draw (0.734615, 0.625000) to (0.734615, 0.610294);
\draw (0.803846, 0.625000) to (0.803846, 0.610294);
\draw (0.700000, 0.558824) to (0.758846, 0.558824);
\draw (0.700000, 0.566176) to (0.700000, 0.551471);
\draw (0.758846, 0.566176) to (0.758846, 0.551471);
\draw (0.779615, 0.500000) to (0.838462, 0.500000);
\draw (0.779615, 0.507353) to (0.779615, 0.492647);
\draw (0.838462, 0.507353) to (0.838462, 0.492647);
\draw (0.745000, 0.441176) to (0.793462, 0.441176);
\draw (0.745000, 0.448529) to (0.745000, 0.433824);
\draw (0.793462, 0.448529) to (0.793462, 0.433824);

\draw (0.898846, 0.617647) to (0.947308, 0.617647);
\draw (0.898846, 0.625000) to (0.898846, 0.610294);
\draw (0.947308, 0.625000) to (0.947308, 0.610294);
\draw (0.853846, 0.558824) to (0.912692, 0.558824);
\draw (0.853846, 0.566176) to (0.853846, 0.551471);
\draw (0.912692, 0.566176) to (0.912692, 0.551471);
\draw (0.933462, 0.500000) to (0.992308, 0.500000);
\draw (0.933462, 0.507353) to (0.933462, 0.492647);
\draw (0.992308, 0.507353) to (0.992308, 0.492647);
\draw (0.888462, 0.441176) to (0.957692, 0.441176);
\draw (0.888462, 0.448529) to (0.888462, 0.433824);
\draw (0.957692, 0.448529) to (0.957692, 0.433824);

\draw (0.084615, 0.676471) to (0.126154, 0.676471);
\draw (0.084615, 0.683824) to (0.084615, 0.669118);
\draw (0.126154, 0.683824) to (0.126154, 0.669118);
\draw (0.146923, 0.676471) to (0.174615, 0.676471);
\draw (0.146923, 0.683824) to (0.146923, 0.669118);
\draw (0.174615, 0.683824) to (0.174615, 0.669118);
\draw (0.209231, 0.676471) to (0.223077, 0.676471);
\draw (0.209231, 0.683824) to (0.209231, 0.669118);
\draw (0.223077, 0.683824) to (0.223077, 0.669118);

\draw (0.238462, 0.676471) to (0.252308, 0.676471);
\draw (0.238462, 0.683824) to (0.238462, 0.669118);
\draw (0.252308, 0.683824) to (0.252308, 0.669118);
\draw (0.290385, 0.676471) to (0.314615, 0.676471);
\draw (0.290385, 0.683824) to (0.290385, 0.669118);
\draw (0.314615, 0.683824) to (0.314615, 0.669118);
\draw (0.342308, 0.676471) to (0.376923, 0.676471);
\draw (0.342308, 0.683824) to (0.342308, 0.669118);
\draw (0.376923, 0.683824) to (0.376923, 0.669118);

\draw (0.392308, 0.676471) to (0.503077, 0.676471);
\draw (0.392308, 0.683824) to (0.392308, 0.669118);
\draw (0.503077, 0.683824) to (0.503077, 0.669118);
\draw (0.516923, 0.676471) to (0.530769, 0.676471);
\draw (0.516923, 0.683824) to (0.516923, 0.669118);
\draw (0.530769, 0.683824) to (0.530769, 0.669118);

\draw (0.546154, 0.676471) to (0.560000, 0.676471);
\draw (0.546154, 0.683824) to (0.546154, 0.669118);
\draw (0.560000, 0.683824) to (0.560000, 0.669118);
\draw (0.573846, 0.676471) to (0.684615, 0.676471);
\draw (0.573846, 0.683824) to (0.573846, 0.669118);
\draw (0.684615, 0.683824) to (0.684615, 0.669118);

\draw (0.700000, 0.676471) to (0.838462, 0.676471);
\draw (0.700000, 0.683824) to (0.700000, 0.669118);
\draw (0.838462, 0.683824) to (0.838462, 0.669118);

\draw (0.853846, 0.676471) to (0.992308, 0.676471);
\draw (0.853846, 0.683824) to (0.853846, 0.669118);
\draw (0.992308, 0.683824) to (0.992308, 0.669118);

\draw (0.084615, 0.382353) to (0.112308, 0.382353);
\draw (0.084615, 0.389706) to (0.084615, 0.375000);
\draw (0.112308, 0.389706) to (0.112308, 0.375000);
\draw (0.133077, 0.382353) to (0.174615, 0.382353);
\draw (0.133077, 0.389706) to (0.133077, 0.375000);
\draw (0.174615, 0.389706) to (0.174615, 0.375000);
\draw (0.195385, 0.382353) to (0.223077, 0.382353);
\draw (0.195385, 0.389706) to (0.195385, 0.375000);
\draw (0.223077, 0.389706) to (0.223077, 0.375000);

\draw (0.238462, 0.382353) to (0.262692, 0.382353);
\draw (0.238462, 0.389706) to (0.238462, 0.375000);
\draw (0.262692, 0.389706) to (0.262692, 0.375000);
\draw (0.290385, 0.382353) to (0.325000, 0.382353);
\draw (0.290385, 0.389706) to (0.290385, 0.375000);
\draw (0.325000, 0.389706) to (0.325000, 0.375000);
\draw (0.352692, 0.382353) to (0.376923, 0.382353);
\draw (0.352692, 0.389706) to (0.352692, 0.375000);
\draw (0.376923, 0.389706) to (0.376923, 0.375000);

\draw (0.392308, 0.382353) to (0.530769, 0.382353);
\draw (0.392308, 0.389706) to (0.392308, 0.375000);
\draw (0.530769, 0.389706) to (0.530769, 0.375000);

\draw (0.546154, 0.382353) to (0.684615, 0.382353);
\draw (0.546154, 0.389706) to (0.546154, 0.375000);
\draw (0.684615, 0.389706) to (0.684615, 0.375000);

\draw (0.700000, 0.382353) to (0.838462, 0.382353);
\draw (0.700000, 0.389706) to (0.700000, 0.375000);
\draw (0.838462, 0.389706) to (0.838462, 0.375000);

\draw (0.853846, 0.382353) to (0.992308, 0.382353);
\draw (0.853846, 0.389706) to (0.853846, 0.375000);
\draw (0.992308, 0.389706) to (0.992308, 0.375000);

\node at (2/13, 1-11.5/17) {\LARGE yes};
\node at (4/13, 1-11.5/17) {\LARGE yes};
\node at (6/13, 1-11.5/17) {\LARGE no};
\node at (8/13, 1-11.5/17) {\LARGE yes};
\node at (10/13, 1-11.5/17) {\LARGE no};
\node at (12/13, 1-11.5/17) {\LARGE no};

\node at (2/13, 1-12.5/17) {\LARGE yes};
\node at (4/13, 1-12.5/17) {\LARGE yes};
\node at (6/13, 1-12.5/17) {\LARGE yes};
\node at (8/13, 1-12.5/17) {\LARGE no};
\node at (10/13, 1-12.5/17) {\LARGE no};
\node at (12/13, 1-12.5/17) {\LARGE no};

\node at (2/13, 1-13.5/17) {\LARGE yes};
\node at (4/13, 1-13.5/17) {\LARGE no};
\node at (6/13, 1-13.5/17) {\LARGE yes};
\node at (8/13, 1-13.5/17) {\LARGE no};
\node at (10/13, 1-13.5/17) {\LARGE yes};
\node at (12/13, 1-13.5/17) {\LARGE no};

\node at (2/13, 1-14.5/17) {\large $3y_1 - 3 x_2 + 3$};
\node at (2/13, 1-15.5/17) {\large $4 y_1 - 2 x_2 + 2$};
\node at (2/13, 1-16.5/17) {\small difference dominant};

\node at (4/13, 1-14.5/17) {\large $3y_1 - 3 x_2 + 3$};
\node at (4/13, 1-15.5/17) {\large $2 y_1 - 4 x_2 + 4$};
\node at (4/13, 1-16.5/17) {\small difference dominant};

\node at (6/13, 1-14.5/17) {\large $y_1 - 2 x_2 + 3$};
\node at (6/13, 1-15.5/17) {\large $2$};
\node at (6/13, 1-16.5/17) {\small difference dominant};

\node at (8/13, 1-14.5/17) {\large $2y_1 - x_2 + 2$};
\node at (8/13, 1-15.5/17) {\large $2$};
\node at (8/13, 1-16.5/17) {\small difference dominant};

\node at (10/13, 1-14.5/17) {\large $2$};
\node at (10/13, 1-15.5/17) {\large $2$};
\node at (10/13, 1-16.5/17) {\small balanced};

\node at (12/13, 1-14.5/17) {\large $2$};
\node at (12/13, 1-15.5/17) {\large $2$};
\node at (12/13, 1-16.5/17) {\small balanced};

\end{tikzpicture}
\captionof{table}[Table]{Table enumerating the structures of each of the six regions of $A$, along with the size of the sumset, difference set, and type.}
\label{table:huge}
\end{figure}
\end{landscape}
\restoregeometry

\bigbreak


\printbibliography

\ \\

\addresseshere


\end{document}